\documentclass[a4,12pt]{article}%

\usepackage{graphicx}
\usepackage{graphics}
\usepackage{amsmath}
\usepackage{amsfonts}
\usepackage{amssymb}
\usepackage{enumerate}
\usepackage{amsmath}
\usepackage{graphicx}
\usepackage{comment}
\usepackage{mathtools}
\usepackage{float}
\usepackage{etoolbox}
\usepackage{natbib}
\usepackage{dsfont}
\usepackage{bm}

\usepackage[utf8]{inputenc}
\usepackage{array}
\newcolumntype{y}[1]{>{\centering\let\newline\\\arraybackslash\hspace{0pt}}p{#1}}

\newtheorem{theorem}{Theorem}[section]

\newtheorem{claim}[theorem]{Claim}

\newtheorem{definition}[theorem]{Definition}

\newtheorem{lemma}[theorem]{Lemma}

\newtheorem{observation}[theorem]{Observation}

\newenvironment{proof}[1][Proof]{\textbf{#1.} }{\ \rule{0.5em}{0.5em}}

\newcommand{\prob}{{\rm \bf P}}

\newcommand{\lcp}{{\rm LCP}}

\newcommand{\dN}{{\mathbb N}}

\newcommand{\dR}{{\mathbb R}}

\newcommand{\ep}{\varepsilon}
\newcommand{\norm}[1]{\left\lVert#1\right\rVert}
\DeclarePairedDelimiter{\abs}{\lvert}{\rvert}
\DeclareMathOperator*{\argmax}{arg\,max}
\newcommand{\cupdot}{\mathbin{\mathaccent\cdot\cup}}
\newcommand{\period}{\text{.}}
\newcommand{\comma}{\text{,}}
\newcommand\restr[2]{{
		\left.\kern-\nulldelimiterspace 
		#1 
		\vphantom{\big|} 
		\right|_{#2} 
}}

\newcounter{figurecounter}
\usepackage{caption}
\begin{document}

\title{Sunspot Equilibrium in Positive Recursive Two-Dimensions Quitting Absorbing Games%
\thanks{The authors acknowledge the support of the Israel Science Foundation, Grant \#217/17.}}
\author{Orin Munk, Eilon Solan%
\thanks{The School of Mathematical Sciences, Tel Aviv
University, Tel Aviv 6997800, Israel. e-mail: orin25@gmail.com, eilonsolan@gmail.com.}}
\maketitle

\begin{abstract}
	A uniform sunspot $\ep$-equilibrium of a dynamic game is a uniform $\ep$-equilibrium in an extended game,
	where the players observe a public signal at every stage.
	We prove that a uniform sunspot $\ep$-equilibrium exists in two classes of multiplayer absorbing games, thereby extending earlier works by Solan and Solan (2019, 2018).
\end{abstract}


\section{Introduction}
Existence of uniform equilibrium in stochastic games
has been a topic of great interest in the last decades.
\cite{Shapley} introduced the model of stochastic games,
in which players' actions affect both the stage payoff and the state of the game.
Shapley proved that every two-player zero-sum stochastic game admits a $\lambda$-discounted equilibrium in stationary strategies, for every $\lambda > 0$.
This result was later extended for multiplayer stochastic games by \cite{fink} and \cite{takahashi_1964}.

A strategy profile is a uniform $\ep$-equilibrium if it is a $\lambda$-discounted $\ep$-equilibrium,
for every discount factor $\lambda$ sufficiently close to $0$.
\cite{mertens_neyman} proved that every two-player zero-sum stochastic game admits a uniform $\ep$-equilibrium, for every $\ep > 0$. It is a strategy profile in which each player cannot gain more then $\ep$ by deviating from it, for every small enough $\lambda$. \cite{vieille_two-player_2000_B} extended this result to every two-player non-zero-sum stochastic game.
\cite{solan_three-player} proved the existence of a uniform $\ep$-equilibrium in three-player absorbing games.
\cite{solan_vieille_2001} presented the class of quitting games,
and proved that if the payoffs satisfy a certain condition,
then a uniform $\ep$-equilibrium exists.
For additional results on the existence of uniform $\ep$-equilibrium in stochastic games,
see \cite{simon_2007, simon_2012, simon_2016}.
To this day it is not known if every $n$-player stochastic game admits a uniform $\ep$-equilibrium.

A uniform correlated $\ep$-equilibrium is a uniform $\ep$-equilibrium in an extended game that includes a correlation device.
Various versions of uniform correlated $\ep$-equilibrium have been studied,
depending on the type of correlation that the device allows.
\cite{solan_vieille_correlated_2002} studied extensive-form correlation devices.
These are devices that send a private message to each player in every stage.
They proved that every multiplayer stochastic game admits a uniform extensive-form correlated $\ep$-equilibrium.
\cite{solan_vohra_correlated_quitting_2001, solan_vohra_correlated_2002} studied a normal-form correlation device,
which sends one private message to each player at the outset of the game.
They proved that every absorbing game\footnotemark admits a uniform normal-form correlated $\ep$-equilibrium.
\footnotetext{An absorbing game is a stochastic game with a single non-absorbing state.}
Recently, \cite{quitting, general_quitting} studied uniform sunspot $\ep$-equilibria,
which are uniform $\ep$-equilibria in an extended game that includes a device,
which sends a public message at every stage.
They proved that every quitting game,
as well as every general quitting game\footnotemark,
admits a uniform sunspot $\ep$-equilibrium.
\footnotetext{A quitting game is an absorbing game where,
	in the non-absorbing state,
each player has two actions, continue and quit.
As long as all players continue,
the play remains in the non-absorbing state;
as soon as at least one player quits, the play moves with probability $1$ to some absorbing state.}

This paper is part of a project whose goal is to prove that every multiplayer stochastic game admits a uniform sunspot $\ep$-equilibrium.
Here we will present two classes of absorbing games, spotted games and L-shaped games.
Both of these classes are generalizations of quitting games.
In quitting games there is a single non-absorbing entry.
An absorbing game is \emph{spotted} if every two non-absorbing entries differ by the actions of at least two players.
An absorbing game is \emph{L-shaped} if there are exactly three non-absorbing entries $a^1$, $a^2$, and $a^3$,
and moreover, the action profiles $a^1$ and $a^3$ differ by the action of Player 1 only,
while the action profiles $a^1$ and $a^2$ differ by the action of Player 2 only.
We will show that every game in these classes admits a uniform sunspot $\ep$-equilibrium.

The paper is organized as follows.
The model and the main results are described in Section \ref{section:model}.
The proof of the main result for the first class presented, spotted games, appears in Section \ref{section:spotted}.
The proof of the main result for the second class presented, L-shaped games, appears in Section \ref{section:L-shaped}.
Discussion appears in Section~\ref{section:discussion}.

\section{Model and Main Results}
\label{section:model}

\subsection{Absorbing Games and General Quitting Games}
In this paper we study stochastic games with a single non-absorbing state. These games, called absorbing games, were first studied by \cite{kohlberg_absorbing}.
\begin{definition}
	An \emph{absorbing game} is a tuple $\Gamma = (I, (A_i)_{i \in I}, P, u)$, where
	\begin{itemize}
		\item   $I=\left\{1,2,...,\abs{I}\right\}$ is a non-empty finite set of players.
		\item   For every player $i$, $A_i$ is a non-empty finite set of actions.
		Denote the set of all action profiles by $A = \times_{i \in I} A_i$.
		\item   $u : A \to [0,1]^I$ is a \emph{payoff function}.
		\item	$P : A \to [0,1]^I$ is an \emph{absorption probability function}.
	\end{itemize}
\end{definition}
The game proceeds as follows:
At every stage $t \in \dN$,
until absorption occurs,
each player $i \in I$ chooses an action $a_i^t \in A_i$,
and receives the payoff $u_i(a^t)$,
where $a^t \coloneqq \times_{i \in I} a_i^t$.
With probability $P(a^t)$, the game is absorbed and the players' payoff is $u(a^t)$ in each subsequent stage.
Otherwise, the game continues to the next stage.\\

An action profile $a \in A$ is \emph{absorbing} if $P(a) > 0$,
and \emph{non-absorbing} if $P(a) = 0$.
For a mixed action profile $x \in \Delta(A)$, we denote
$P(x) \coloneqq \sum_{a \in A} x(a) \cdot P(a)$.

\cite{general_quitting} defined a class of absorbing games, called general quitting games.
In this class of games, each player has two types of actions:
continue actions and quitting actions,
and the game is absorbed as soon as at least one player plays a quitting action.
\begin{definition}
	An absorbing game $\Gamma = (I, (A_i)_{i \in I}, P, u)$ is a \emph{general quitting game}
	if for every player $i \in I$, the action set $A_i$ can be divided into two non-empty finite disjoint sets $A_i = C_i \cupdot Q_i$ that satisfy the following condition:
	for every action profile $a \in A$, if $a_j \in Q_j$ for some $j \in I$, then $P(a) = 1$.
	Otherwise, $a \in \times_{i \in I} C_i$ and $P(a) = 0$.
\end{definition}
We will denote a general quitting game by $\Gamma = (I, (C_i)_{i \in I}, (Q_i)_{i \in I}, P, u)$.\\
For every $i \in I$, the actions in $C_i$ are called \emph{continue} actions, and the actions in $Q_i$ are called \emph{quitting} actions.

\begin{definition}
	An absorbing game is called \emph{generic} if every two distinct action profiles $a \ne a' \in A$ yield a different payoff for every player;
	that is, $u_i(a) \ne u_i(a')$ for every $i \in I$.
\end{definition}

\begin{definition}
	An absorbing game $\Gamma = (I, (A_i)_{i \in I}, P, u)$ is \emph{recursive} if every non-absorbing action profile yields payoff 0;
	that is, for every $a \in A$, if $P(a) = 0$ then $u(a) = \vec{0}$.
\end{definition}

\begin{definition}
	A recursive absorbing game $\Gamma = (I, (A_i)_{i \in I}, P, u)$ is \emph{positive} if the payoff function is non-negative; that is, for every $a \in A$ we have $u_i(a) \ge 0$.
\end{definition}

A special class of general quitting games is quitting games, previously studied by \cite{cyclic_1997} and \cite{solan_vieille_2001}.
\begin{definition}
	A general quitting game $\Gamma = (I, (C_i)_{i \in I}, (Q_i)_{i \in I}, P, u)$ is a \emph{quitting game} if $\abs{C_i} = \abs{Q_i} = 1$ for every $i \in I$.
\end{definition}

\subsection{Strategies and Payoff}
\begin{definition}
	Let $\Gamma = (I, (A_i)_{i \in I}, P, u)$ be an absorbing game.
	A \emph{strategy} of player $i$ in $\Gamma$ is a function
	$\sigma_i : \cup_{t \in \dN} A^{t-1} \to \Delta(A_i)$.
	A \emph{strategy profile} is a vector $\sigma = \left( \sigma_i \right)_{i \in I}$
	of strategies,
	one for each player.
\end{definition}

\begin{definition}
	A strategy $\sigma_i$ of player $i$ is called \emph{stationary} if it depends solely on the current stage, and not on past play. Since absorbing games have only one non-absorbing state,
	a stationary strategy of player $i$ is equivalent to a probability distribution in $\Delta(A_i)$.
	A strategy profile $\sigma$ is called \emph{$\ep$-almost stationary} if it is a stationary strategy, supplemented with threats of punishment.
\end{definition}

We will study the concept of uniform equilibrium.
To this end, we present the $T$-stage payoff, the discounted payoff, and the undiscounted payoff.

\begin{definition}
	Let $T \in \dN$, and let $\Gamma = (I, (A_i)_{i \in I}, P, u)$ be an absorbing game.
	The \emph{$T$-stage payoff} of player $i$ in $\Gamma$ under strategy profile $\sigma$ is given by
	$$\gamma_i^T(\sigma) \coloneqq \frac{1}{T} \mathbb{E}_\sigma \left[ \sum_{t=1}^{T} u_i(a^t) \right] \period$$
\end{definition}

\begin{definition}
	Let $\lambda \in (0,1]$, and let $\Gamma = (I, (A_i)_{i \in I}, P, u)$ be an absorbing game.
	The \emph{$\lambda$-discounted payoff} of player $i$ in $\Gamma$ under strategy profile $\sigma$ is given by
	$$\gamma_i^\lambda(\sigma) \coloneqq \mathbb{E}_\sigma \left[ \sum_{t=1}^{\infty} \lambda (1 - \lambda)^{t-1} u_i(a^t) \right] \period$$
\end{definition}

\begin{definition}
	\label{definition:ep-eq}
	Let $\ep > 0$ and let $\Gamma = (I, (A_i)_{i \in I}, P, u)$ be an absorbing game.
	A strategy profile $\sigma$ is an \emph{$\ep$-equilibrium}
	if no player can increase her payoff by more then $\ep$ by deviating from $\sigma$.
	That is, for every player $i \in I$ and every strategy $\sigma_i'$ of player $i$,
	$$\gamma_i^\lambda(\sigma'_i, \sigma_{-i}) \le \gamma_i^\lambda(\sigma) + \ep \period$$
\end{definition}
If this payoff function is a $\lambda$-discounted payoff for some $\lambda \in [0,1]$, then this equilibrium is called \emph{$\lambda$-discounted $\ep$-equilibrium}.
If this payoff function is a $T$-stage payoff for some $T \in \dN$, then this equilibrium is called \emph{$T$-stage $\ep$-equilibrium}.\\
\cite{fink} proved  that for every $\lambda \in (0,1)$,
every stochastic game, and in particular, every absorbing game,
admits a $\lambda$-discounted $0$-equilibrium in stationary strategies.\\
The concept of uniform $\ep$-equilibrium was introduced by \cite{mertens_neyman}.

\begin{definition}
	\label{definition:uniform}
	Let $\ep>0$.
	A strategy profile $\sigma$ is a \emph{uniform $\ep$-equilibrium}
	if it satisfies the following two conditions:
	\begin{itemize}
		\item[($\iota$)]	There exists $\lambda_0 \in (0,1]$ such that
		$\sigma$ is a $\lambda$-discounted $\ep$-equilibrium
		for every $\lambda \in (0,\lambda_0)$.
		\item[($\iota \iota$)]	There exists integer $T_0 \in \dN$ such that
		$\sigma$ is a $T$-stage $\ep$-equilibrium
		for every $T \ge T_0$.
	\end{itemize}
\end{definition}
It follows from \cite{absorbing} that every two-player non-zero sum absorbing game admits a uniform $\ep$-equilibrium, for every $\ep > 0$. \cite{solan_three-player} extended this result to every three-player absorbing game.
To date it is not known whether every four-player absorbing game admits a uniform $\ep$-equilibrium,
for every $\ep > 0$.

\begin{definition}
	Let $\Gamma = (I, (A_i)_{i \in I}, P, u)$ be a positive recursive absorbing game.
	The \emph{undiscounted payoff} of player $i$ in $\Gamma$ under action profile strategy $\sigma$,
	is defined by
	$$\gamma_i(\sigma) \coloneqq
	\lim_{t \to \infty} \mathbb{E}_\sigma \left[ u_i(a^t) \right] \coloneqq
	\sum_{a \in A} u_i(a) \cdot P(a^t = a \mid \sigma) 	\period$$
\end{definition}
The equilibrium described in Definition \ref{definition:ep-eq} is called \emph{undiscounted $\ep$-equilibrium}
if the payoff function used in the definition is the undiscounted payoff.
It follows from \cite{solan_vieille_2001} that there is an equivalence between uniform $\ep$-equilibrium and undiscounted $\ep$-equilibrium in positive recursive absorbing games.\\

The following lemma allows us to focus on generic games.
Though the lemma is stated for the concept of uniform $\ep$-equilibrium,
it applies to all notions of equilibrium that are mentioned in the paper.

\begin{lemma}
	\label{lemma:generic}
	Let $\ep > 0$,
	and let
	$\Gamma = (I, (C_i)_{i \in I}, (Q_i)_{i \in I}, P, u)$ and $\Gamma' = (I, (C_i)_{i \in I}, (Q_i)_{i \in I}, P, u')$ be two absorbing games with the same absorption probability function.
	If $\norm{u - u'}_\infty \le \ep$,
	then a uniform $\ep$-equilibrium of $\Gamma$
	is a uniform $3\ep$-equilibrium of $\Gamma'$.
\end{lemma}

\subsection{Sunspot Equilibrium}
We extend the game $\Gamma$ by introducing a public correlation device:
at the beginning of every stage $t \in \dN$ the players observe a public signal $\zeta^t \in [0,1]$ that is drawn according to the uniform distribution,
independently of past signals and play.
The extended game is denoted by $\Gamma^{sun}$.
A \emph{strategy} of player $i$ in the extended game $\Gamma^{sun}$ is a measurable function
$\xi_i : \cup_{t \in \dN} ([0,1]^{t} \times A^{t-1}) \to \Delta(A_i)$.

\begin{definition}
	Every uniform $\ep$-equilibrium of the game $\Gamma^{sun}$ is called a \emph{uniform sunspot $\ep$-equilibrium} of $\Gamma$.
\end{definition}
The main goal of this paper is to study uniform sunspot $\ep$-equilibrium in absorbing games.

\subsection{Uniform Sunspot Equilibrium in General Quitting Games}
\cite{general_quitting} proved the existence of a uniform sunspot $\ep$-equilibrium in general quitting games in which every player has a single quitting action. This result relies on linear complementarity problems, which we present now.

\begin{definition}
	Given an $n\times n$ matrix $R$ and a vector $q \in \dR^n$, the \emph{linear complementarity problem $LCP(R, q)$} is the following problem:
	\begin{eqnarray*}
		\hbox{Find}&&w \in \dR^n_+, z \in \Delta(\left\{0,1,\dots,n\right\}) \nonumber\\
		\hbox{such that}
		&&w = z_0 \cdot q + R \cdot (z_1\dots,z_n)^\intercal, \label{lpc1}\\
		&&z_i = 0 \hbox{ or } w_i = R_{i,i} \ \ \ \forall i \in \left\{1,2,\dots,n\right\}.
		\nonumber
	\end{eqnarray*}
\end{definition}

\begin{definition}
	An $n\times n$ matrix $R$ is called a \emph{$Q$-matrix} if for every $q \in \dR^n$ the linear complementarity problem $\lcp(R,q)$ has at least one solution.
\end{definition}

\begin{definition}
	If an $n\times n$ matrix $R$ is not a $Q$-matrix,
	then a vector $q \in \dR^n$ such that the linear complementarity problem $\lcp(R,q)$ has no solution, is called a \emph{witness} of $R$.
\end{definition}

\begin{definition}
\label{definition:original best response}
Let $\Gamma = (I,(C_i)_{i \in I}, (Q_i)_{i \in I}, P, u)$ be a general quitting game,
such that each player has a single quitting action;
that is, $\abs{Q_i} = 1$ for every $i \in I$.
For every continue action profile $c \in C$,
denote by $R(\Gamma, c)$ the $(|I|\times|I|)$-matrix whose $i$'th column is $u(q_i,c_{-i})$.
\end{definition}

\begin{theorem}[Solan and Solan, 2018]
	\label{theorem:Solan and Solan}
	Let $\Gamma = (I,(C_i)_{i \in I}, (Q_i)_{i \in I}, P, u)$ be a positive recursive general quitting game, where $Q_i = \left\{q_i\right\}$ for every player $i$.
	\begin{itemize}
		\item
		If the matrix $R(\Gamma, c)$ is not a $Q$-matrix for every action profile $c \in \times_{i 
		\in I} \Delta(C_i)$ of continue actions, then there is an absorbing stationary strategy profile $x$
			such that for every $\ep > 0$ the stationary strategy $x$, supplemented with threats of punishment, is a uniform $\ep$-equilibrium of $\Gamma$.
		\item
		If the matrix $R(\Gamma, c)$ is a $Q$-matrix for some action profile $c \in \times_{i \in I} \Delta(C_i)$ of continue actions,
		then for every $\ep > 0$ the game $\Gamma$ admits a uniform sunspot $\ep$-equilibrium $\xi$ in which,
		after every finite history, at most one player $i$ quits with positive probability,
		and does so with probability at most $\ep$,
		while all other players play $c_{-i}$.
	\end{itemize}
\end{theorem}
Theorem \ref{theorem:Solan and Solan} is valid in the case where some players have no quitting actions, as well as when the absorbing probabilities are smaller then 1, that is, $P(a) \in (0,1]$ whenever $a_i \in Q_i$ for at least one player $i \in I$.

\subsection{Quitting Absorbing Games}
Quitting absorbing games are absorbing games where at least one player has a quitting action.
\begin{definition}
	\label{definition:quitting absorbing game}
	An absorbing game $\Gamma = (I, (A_i)_{i \in I}, P, u)$ is called a \emph{quitting absorbing game}
	if for every $i \in I$, the set $A_i$ can be divided into two disjoint sets of \emph{continue actions} and \emph{quitting actions}, $A_i = C_i \cupdot Q_i$ such that
	\begin{itemize}
		\item	For every quitting action $q_i \in Q_i$ and for every action profile $a_{-i} \in A_{-i}$, $P(q_i, a_{-i}) > 0$.
		\item	For every continue action $c_i \in c_i$ there is a continue action profile $c_{-i} \in \times_{j \ne i} C_{j}$ such that $P(c_i, c_{-i}) = 0$.
		\item	There is at least one quitting action, namely $\cup_{i \in I} Q_i \ne \phi$.
	\end{itemize}
\end{definition}
We denote a quitting absorbing game by $\Gamma = (I, (C_i)_{i \in I}, (Q_i)_{i \in I}, P, u)$.
In contrast to general quitting games, where the game continues with probability 1 as soon as no player plays a quitting action, in quitting absorbing games the game may absorb in this case.

\subsection{Two-Dimensions Quitting Absorbing Games}
\label{section:two dimension}
We here define a simple class of quitting absorbing games, where two players have two continue actions, while all other players have a single continue action.

\begin{definition}
	A \emph{two-dimension quitting absorbing game} is a quitting absorbing game $\Gamma = (I, (C_i)_{i \in I}, (Q_i)_{i \in I}, P, u)$ that satisfies $|C_1|=|C_2|=2$ and $|C_i|=1$ for every $i > 2$.
\end{definition}

In two-dimension quitting absorbing games, up to equivalences, there are 6 possible absorption structures for action profiles in $\times_{i \in I} C_i$:
\begin{itemize}
	\item All action profiles are absorbing. In this case the game is equivalent to a one-shot game.
	\item There are three absorption structures for which the game is a general quitting game, see Figure \ref{table:general}.
	\item There are two additional absorption structures, see Figures \ref{table:L-shaped and spotted}.
\end{itemize}
If the absorption structure of $\Gamma$ is equivalent to that in Figure \ref{table:L-shaped and spotted} (left), then the game is called \emph{L-shaped}.
If the absorption structure of $\Gamma$ is equivalent to that in Figure \ref{table:L-shaped and spotted} (right), then the game is called \emph{spotted}.\\

\begin{table}[h!]
	\begin{center}
	\begin{tabular}{c|| c| c| c} 
		 & C & C & Q \\ [0.5ex] 
		\hline \hline
		C &  &  & * \\ 
		\hline
		C &  &  & * \\
		\hline
		Q & * & * & * \\
	\end{tabular}
	\quad
	\begin{tabular}{c|| c| c| c} 
	& C & C & Q \\ [0.5ex] 
	\hline \hline
	C &  &  & * \\ 
	\hline
	C & * & * & * \\
	\hline
	Q & * & * & * \\
\end{tabular}
	\quad
\begin{tabular}{c|| c| c| c} 
	& C & C & Q \\ [0.5ex] 
	\hline \hline
	C &  & * & * \\ 
	\hline
	C & * & * & * \\
	\hline
	Q & * & * & * \\
\end{tabular}
\caption{Two-dimension quitting absorbing games that are general quitting games.}
\label{table:general}
\end{center}
\end{table}

\begin{table}[h!]
	\begin{center}
		\begin{tabular}{c|| c| c| c} 
			& C & C & Q \\ [0.5ex] 
			\hline \hline
			C &  &  & * \\ 
			\hline
			C &  & * & * \\
			\hline
			Q & * & * & * \\
		\end{tabular}
	\quad
	\quad
	\quad
		\begin{tabular}{c|| c| c| c} 
			& C & C & Q \\ [0.5ex] 
			\hline \hline
			C & * &  & * \\ 
			\hline
			C &  & * & * \\
			\hline
			Q & * & * & * \\
		\end{tabular}
		\caption{L-shaped game (left) and spotted game (right).}
		\label{table:L-shaped and spotted}
	\end{center}
\end{table}

\begin{definition}
	An \emph{L-shaped game} is a two-dimension quitting absorbing game $\Gamma = (I, (C_i)_{i \in I}, (Q_i)_{i \in I}, P, u)$ with only one action profile $a \in \times_{i \in I} C_i$ such that $P(a) > 0$.
	\end{definition}
	The game in Figure \ref{table:L-shaped and spotted} (left) is an L-shaped game.

\begin{definition}
	A \emph{spotted game} is an absorbing game $\Gamma = (I, A = \times_{i \in I} A_i, P, u)$ such that every two distinct non-absorbing action profiles differ by at least two coordinates. That is, if $a \ne a' \in A$ are two non-absorbing action profiles, then there are $i \ne j \in I$ such that $a_i \ne a'_i$ and $a_j \ne a'_j$.
	\end{definition}
	The game in Figure \ref{table:L-shaped and spotted} (right) is a spotted game.
We note that the definition of spotted games relates to every absorbing game, and not necessarily to two-dimension quitting absorbing games.

\subsection{Main Results}
The main results of the paper extend the existence of a uniform sunspot $\ep$-equilibrium to the two new classes of games, defined in Section \ref{section:two dimension}.

\begin{theorem}
	\label{theorem:spotted}
	Every positive recursive spotted game admits a uniform sunspot $\ep$-equilibrium, for every $\ep > 0$.
\end{theorem}
Section~\ref{section:spotted} is dedicated to the proof of Theorem \ref{theorem:spotted}.

\begin{theorem}
	\label{theorem:main}
	Every positive recursive two-dimension quitting absorbing game admits a uniform sunspot $\ep$-equilibrium, for every $\ep > 0$.
\end{theorem}

In view of \cite{general_quitting} and Theorem \ref{theorem:spotted}, to prove Theorem \ref{theorem:main} it is sufficient to prove the existence of a uniform sunspot $\ep$-equilibrium in L-shaped two-dimension quitting absorbing games. This will be done in Section \ref{section:L-shaped}.

\section{Spotted Games}
\label{section:spotted}
In this section we prove Theorem \ref{theorem:spotted}. The proof is an adaptation of the proof of \cite{quitting}.
By Lemma \ref{lemma:generic},
to prove that an absorbing game admits a sunspot $\ep$-equilibrium, we can assume without loss of generality that the game is generic.\\

\textbf{Step 1: Constructing best response matrices}\\
Let $\Gamma = (I, (C_i)_{i \in I}, (Q_i)_{i \in I}, P, u)$ be a generic spotted game.
For every non-absorbing action profile $a \in A$ and every player $i \in I$,
the \emph{best absorbing deviation} of player $i$ against $a$ is
$b_i(a) \coloneqq \argmax_{a_i' \ne a_i}{u(a_i', a_{-i})}$.
Since the game is spotted, all deviations of each player $i$ from $a$ are absorbing.
Since the game is generic, the best absorbing deviation is uniquely defined.
Let $r^i(a) \coloneqq u(b_i(a), a_{-i})$. This is the payoff vector when player $i$ optimally deviates from $a$.
Let $R(a)$ be the $\abs{I} \times \abs{I}$ matrix whose $i$-th column is $r^i(a)$.
One of the following two conditions must hold:
\begin{enumerate}
	\item[(E.1)] $R(a)$ is a Q-matrix for some non-absorbing action profile $a \in A$.
	\item[(E.2)] $R(a)$ is not a Q-matrix for every non-absorbing action profile $a \in A$.
\end{enumerate}

\textbf{Step 2: Case (E.1) yields a uniform sunspot $\ep$-equilibrium}\\
Let $a' \in A$ be such that $R(a')$ is a Q-matrix.
We show that in this case a sunspot $\ep$-equilibrium exists for every $\ep > 0$.
Construct an auxiliary general quitting game $\Gamma(a') = (I, (C'_i)_{i \in I}, (Q'_i)_{i \in I}, P', u)$,
where $C'_i = \left\{a'_i\right\}$,
$Q'_i = A_i \setminus \left\{a'_i\right\}$,
and
\[
P'(a) \coloneqq \begin{dcases*}
P(a)  & if $P(a) > 0$,\\
0 & if $a = a'$,\\
1 & otherwise.
\end{dcases*}
\]
By Theorem \ref{theorem:Solan and Solan}, the game $\Gamma(a')$ admits a uniform sunspot $\ep$-equilibrium $\xi$, where all players play the action profile $a' \in A$, and in each stage only one player plays the best absorbing deviation with a positive probability.
The reader can verify that the strategy profile $\xi$ is a uniform sunspot $\ep$-equilibrium in $\Gamma$ as well, since the best deviations from $a'$ in both games $\Gamma$ and $\Gamma(a')$ are the same.\\

\textbf{Step 3: Case (E.2) yields a stationary uniform $\ep$-equilibrium}\\
The condition implies that for every non-absorbing action profile $a \in A$ there is a vector $q^a \in \mathbb{R}^I$ such that the linear complementarity problem $\lcp(R(a), q^a)$ has no solution.
Let $\widehat{\Gamma} = (I, (C_i)_{i \in I}, (Q_i)_{i \in I}, P, \widehat{u})$ be the auxiliary spotted game that is identical to $\Gamma$,
except for the payoff at non-absorbing action profiles:
\[
\widehat{u}(a) \coloneqq \begin{dcases*}
u(a)  & if $P(a) > 0$,\\
q^a & if $P(a) = 0$.
\end{dcases*}
\]
For every $\lambda \in (0,1]$ let $x^\lambda$ be a stationary $\lambda$-discounted equilibrium of $\widehat{\Gamma}$, such that the limit $x^0 \coloneqq \lim_{\lambda \to 0} x^\lambda$ exists.
As in \cite{quitting}, because $R(a)$ is not a Q-matrix for every non-absorbing action profile,
$x^0$ must be absorbing. Moreover, $x^0$ is a stationary uniform $\ep$-equilibrium for every $\ep > 0$.

\section{The L-Shaped Game}
\label{section:L-shaped}
In this section we prove the following result, which, together with Theorem \ref{theorem:spotted} and Theorem \ref{theorem:Solan and Solan}, implies Theorem \ref{theorem:main}.
\begin{lemma}
\label{lemma:L-shaped main}
	Every positive recursive L-shaped game admits a uniform sunspot $\ep$-equilibrium, for every $\ep>0$.
\end{lemma}
Throughout the section we will denote the two continue actions of Players 1 and 2 by
$C_1 = \left\{c_1^1,c_1^2\right\}$, $C_2 = \left\{c_2^1,c_2^2\right\}$, respectively,
and assume that $P(c_1^2, c_2^2, c_3,..., c_{|I|}) > 0$.
We will use the following notations:
\begin{eqnarray*}
	a^1 \coloneqq (c_1^1, c_2^1, c_3,...c_{|I|})\comma
	& &a^2 \coloneqq (c_1^1, c_2^2, c_3,...c_{|I|})\comma\\
	a^3 \coloneqq (c_1^2, c_2^1, c_3,...c_{|I|})\comma
	& \text{and} &a^4 \coloneqq (c_1^2, c_2^2, c_3,...c_{|I|}) \period
\end{eqnarray*}

\begin{table}[h!]
	\begin{center}
		\begin{tabular}{c|| c| c| c} 
			& $c_2^1$ & $c_2^2$ & $q_2$ \\ [0.5ex] 
			\hline \hline
			$c_1^1$ & $a^1$ & $a^2$ & * \\ 
			\hline
			$c_1^2$ & $a^3$ & $a^4$ * & * \\
			\hline
			$q_1$ & * & * & * \\
		\end{tabular}
		\caption{An L-shaped game.}
		\label{table:L-shaped a1a2a3a4}
	\end{center}
\end{table}

\subsection{Auxiliary Games}
We start by introducing a collection of auxiliary games that are derived from the L-shaped game by turning some non-absorbing action profiles into absorbing action profiles.

\begin{definition}
	\label{def:auxiliary}
	Let $\delta_1, \delta_2 \in [0,1]$ and let $\Gamma = (I, (C_i)_{i \in I}, (Q_i)_{i \in I}, P, u)$ be an L-shaped game.
	Let $\Gamma^{\delta_1,\delta_2} = (I, (C_i)_{i \in I}, (Q_i)_{i \in I}, P^{\delta_1, \delta_2}, u^{\delta_1, \delta_2})$ be the quitting absorbing game that is defined as follows (see Figure \ref{table:L-shaped and auxiliary}):
	\begin{itemize}
		\item $P^{\delta_1, \delta_2}(a^2) \coloneqq \delta_2$
				and $P^{\delta_1, \delta_2}(a^3) \coloneqq \delta_1$.
		\item If $\delta_2 > 0$ then $u^{\delta_1, \delta_2}(a^2) \coloneqq u(a^4)$.
		Otherwise $u^{\delta_1, \delta_2}(a^2) \coloneqq u(a^2)$.
		\item If $\delta_1 > 0$ then $u^{\delta_1, \delta_2}(a^3) \coloneqq u(a^4)$.
		Otherwise $u^{\delta_1, \delta_2}(a^3) \coloneqq u(a^3)$.
		\item For every action profile $a \in A$ such that $a \ne a^2, a^3$, we set
		$P^{\delta_1, \delta_2}(a) \coloneqq P(a)$
		and $u^{\delta_1, \delta_2}(a) \coloneqq u(a)$.
	\end{itemize}
\end{definition}
In other words, the auxiliary game $\Gamma^{\delta_1,\delta_2}$ is similar to $\Gamma$,
but we turn one or two action profiles to absorbing,
with absorbing payoff that is equal to $u(a^4)$.
The auxiliary game $\Gamma^{\delta_1,\delta_2}$ is a quitting game if $\delta_1, \delta_2 > 0$ and a general quitting game if $\max\left\{\delta_1, \delta_2\right\} > 0$.\\

\begin{table}[h!]
	\begin{center}
		\begin{tabular}{c|| c| c| c} 
			& $c_2^1$ & $c_2^2$ & $q_2$ \\ [0.5ex] 
			\hline \hline
			$c_1^1$ &  &  & * \\ 
			\hline
			$c_1^2$ &  & $u(a^4)$ * & * \\
			\hline
			$q_1$ & * & * & * \\
		\end{tabular}
		\quad
		\begin{tabular}{c|| c| c| c} 
			& $c_2^1$ & $c_2^2$ & $q_2$ \\ [0.5ex] 
			\hline \hline
			$c_1^1$ &  & $u(a^4)\ ^{\delta_2}$* & * \\ 
			\hline
			$c_1^2$ & $u(a^4)\ ^{\delta_1}$* & $u(a^4)$ * & * \\
			\hline
			$q_1$ & * & * & * \\
		\end{tabular}
		\caption{An L-shaped game $\Gamma$ (left)
			and the auxiliary game $\Gamma^{\delta_1, \delta_2}$ (right).}
		\label{table:L-shaped and auxiliary}
	\end{center}
\end{table}

We will denote the min-max value of player $i$ in the absorbing game $\Gamma$ by
$\overline{v}_i(\Gamma) \coloneqq \min_{\sigma_i} \max_{\sigma_{-i}} \gamma_i(\sigma_i, \sigma_{-i})$
This is a lower bound on player $i$'s equilibrium payoff.
We will prove that the min-max value of each player in the auxiliary game is not lower by much than her min-max value in the original game.
To this end we need the following result.

\begin{lemma}
	\label{lemma:pre-minmax}
	Let $\Gamma = (I, (A_i)_{i \in I}, u, P)$ be a positive recursive quitting absorbing game.
	Fix a player $i \in I$, such that $\abs{Q_i} > 0$.
	For every $\ep > 0$ there exists $T_\ep \in \dN$ such that
	for every strategy profile $x_{-i}$
	there is a pure strategy $a_i$ of player $i$ such that
	\begin{itemize}
		\item   The payoff under $(a_i,x_{-i})$ is high:
		$\gamma_i(a_i,x_{-i}) \geq \overline{v}_i(\Gamma) - \ep$.
		\item
		The probability of absorption up to time $T_\ep$ under $(a_i,x_{-i})$ is high:
		$\prob_{a_i.x_{-i}}(\Gamma \text{ is absorbed in the first } T_\ep \text{ stages}) \geq 1-\ep$.
	\end{itemize}
\end{lemma}

\begin{proof}
	Since the game is positive recursive,
	and since player $i$ can obtain a positive payoff by quitting,
	then the min-max value of player $i$ is positive, that is,
	$\overline{v}_i(\Gamma) > 0$.
	Fix $\ep > 0$ sufficiently small.\\
	Let $x_{-i} \in \times_{j \ne i} C_j$ be a strategy profile such that under $(a_i,x_{-i})$ the play never absorbs,
	for every continue action action $c_i \in C_i$ of player $i$.
	Since $\overline{v}_i(\Gamma) > 0$, it follows that
	there is a quitting action $q_i \in Q_i$
	such that $\gamma_i(q_i,x_{-i}) \geq \overline{v}_i(\Gamma)$.
	It follows that $\gamma_i(q_i,x'_{-i}) \geq \overline{v}_i(\Gamma) - \ep$ for every strategy profile $x'_{-i}$ that satisfies $\norm{x_{-i}-x'_{-i}}_{\infty} \leq \ep$.
	We deduce that there is a quitting action $q_i \in Q_i$ satisfies the requirements with $T_\ep=1$,
	for every strategy profile $x'_{-i}$ for which $\prob(\text{absorption} \mid (a_i,x'_{-i})) \leq \ep$ for every $a_i \in A_i$.\\
	Let $x_{-i}$ be a strategy profile such that for some continue action $c_i \in C_i$ of player $i$, under $(c_i,x_{-i})$ the game absorbs.
	Then one of the following conditions holds:
	\begin{itemize}
		\item   Player $i$ has a quitting action $q'_i \in Q_i$ such that $\gamma_i(q_i,x_{-i}) \geq \overline{v}_i(\Gamma) -\ep$.
		\item	For every quitting action $q_i \in Q_i$ of player $i$, $\gamma_i(q_i,x_{-i}) < \overline{v}_i(\Gamma) -\ep$ and there is a continue action $c_i \in C_i$ such that $\prob(\text{absorption} \mid c_i,x_{-i}) \geq \ep$.
	\end{itemize}
	In the first case, it follows that the quitting action $q'_i \in Q_i$ satisfies the requirements with $T_\ep=1$ against $x_{-i}$.
	In the latter case, it follows that $\gamma_i(a_i,x_{-i}) \geq \overline{v}_i(\Gamma)$.
	We deduce that the continue action $c_i$ satisfies the requirements with $T_\ep=\tfrac{1}{\ep^2}$.
\end{proof}

\begin{lemma}
	\label{lemma:minmax}
	Let $\Gamma$ be a quitting absorbing game.
	Then for every $\ep>0$,
	there are $\delta_1', \delta_2'> 0$,
	such that for every $\delta_1 \in [0,\delta_1')$, $\delta_2 \in [0,\delta_2')$,
	and every player $i$ we have:
	\[
	\overline{v}_i(\Gamma^{\delta_1,\delta_2}) \ge \overline{v}_i(\Gamma) - \ep.
	\]
\end{lemma}
\begin{proof}
	Let $i \in I$.
	Since the game is positive and recursive, if $\abs{Q_i} = 0$ then
	$$ \overline{v}_i(\Gamma^{\delta_1,\delta_2}) \ge 0 = \overline{v}_i(\Gamma) \period$$
	Then assume $\abs{Q_i} > 0$.
	Let $\ep > 0$ and choose $\ep ' \leq \frac{\ep}{4}$.
	Let $\sigma_i$ and $T$ be, respectively, the strategy and the integer given by Lemma~\ref{lemma:pre-minmax}, with respect to player $i$ and $\ep '$ in the game $\Gamma$.
	Denote $\delta' = \frac{\ep'}{2T}$, and fix $\delta_1,\delta_2 < \delta'$.
	Hence $T\cdot (\delta_1 + \delta_2) \leq \ep '$.\\
	If player $i$ follows $\sigma_i$ in the game $\Gamma^{\delta_1,\delta_2}$, then with probability larger than $1 - \ep ' - T\cdot \delta_1 - T\cdot \delta_2$,
	the game is absorbed in the first $T$ stages by an action profile that is absorbing in $\Gamma$ as well.
	Denote by $\gamma^{\delta_1, \delta_2}_i(\sigma)$ the payoff of player $i$ in $\Gamma^{\delta_1, \delta_2}$ under $\sigma$.
	Then, since payoffs are bounded by $1$,
	\begin{eqnarray*}
	\gamma^{\delta_1, \delta_2}_i(\sigma_i, \sigma_{-i}) &\ge& (1 - \epsilon ' - T\cdot \delta_1 - T\cdot \delta_2) \cdot \gamma_{i}(\sigma_i, \sigma_{-i} \mid\text{absorption at the first T stages})\\
	&\ge& (1 - \ep ' - T\cdot (\delta_1 + \delta_2)) \cdot (\overline{v}_i(\Gamma) - 2\cdot \ep ')\\
	&\ge& \overline{v}_i(\Gamma) - \ep ' - \frac{\ep}{4} - 2\cdot \ep '\\
	&\ge& \overline{v}_i(\Gamma) - 4 \cdot \frac{\ep}{4} \\
	&=& \overline{v}_i(\Gamma) - \ep	
	\end{eqnarray*}
\end{proof}

\subsection{Characterization of L-shaped games}
In this section we divide L-shaped games into two classes of games.
This partition is analogous to the one given in Theorem \ref{theorem:Solan and Solan}.
We start by generalize the definition of best response matrix, presented in Definition \ref{definition:original best response}.

\begin{definition}
	\label{definition:best response matrix}
	Let $\Gamma = (I,(C_i)_{i \in I},(Q_i)_{i \in I},u)$ be a general quitting game, and let $c \in \times_{i \in I} \Delta(C_I)$ be a profile of mixed continue actions.
	The matrix $R$ is a \emph{best response matrix} to $c$ if $R$ is a $(|I|\times|I|)$-matrix whose $i$'th column is $u(q_i,c_{-i})$, where $q_i \in \Delta(Q_i)$ is a best quitting response to $c_{-i}$, that is, $q_i \in \argmax_{\Delta(Q_i)} u_i(q_i,c_{-i})$.
	Denote by $\mathcal{R}(\Gamma, c)$ the set of all best response matrices to the mixed action profile $c$ in the game $\Gamma$.
\end{definition}

\begin{observation}
	\label{observarion:best response matrix}
	Let $\delta_1, \delta_2 \in [0,1)$, and let $\alpha_1, \alpha_2 \in (0,1)$.
	If the matrix $R$ is a best response matrix of $\Gamma^{\delta_1, \delta_2}$,
	then $R$ is also a best response matrix of $\Gamma^{\alpha_1 \delta_1, \alpha_2 \delta_2}$.
\end{observation}

\begin{definition}
	\label{definition:dQ}
	An L-shaped game is called a \emph{QL} game
	if there are $\delta_1, \delta_2 \in [0,1)$ such that $\max\left\{\delta_1, \delta_2\right\} > 0$ and at least one of the best response matrices of $\Gamma^{\delta_1, \delta_2}$ is a Q-matrix.
\end{definition}

\begin{definition}
	\label{definition:CNQ}
	An L-shaped game is called an \emph{NQL} (non-Q L-shaped) game
	if the following sets do not contain a Q-matrix:
	\begin{itemize}
		\item[(NQ1)] $\mathcal{R}(\Gamma^{1,0}, c)$
		\item[(NQ2)] $\mathcal{R}(\Gamma^{0,1}, c)$
		\item[(NQ3)] $\mathcal{R}(\Gamma^{1,1}, c)$
	\end{itemize}
	where $c = (c_1^1, c_2^1, c^3, \dots, c^{\abs{I}} )$.
\end{definition}
Obviously, every L-shaped game is either a QL game, an NQL game, or both.
Section \ref{section:Q} will be dedicated to discuss QL games, while section \ref{section:non q} will be dedicated to discuss NQL games.

\subsection{QL Games}
\label{section:Q}
In this section we prove that QL games admit uniform sunspot $\ep$-equilibrium, for every $\ep > 0$.
The proof is similar to the proof of the analogous result for general quitting games,
which was stated as Theorem \ref{theorem:Solan and Solan}.
\begin{lemma}
	\label{lemma:Q sunspot}
	Let $\Gamma = (I, (C_i)_{i \in I}, (Q_i)_{i \in I}, P, u)$ be a QL game.
	Then for every $\delta_1, \delta_2 \in (0, 1]$, at least one of the games
	$\Gamma^{\delta_1, \delta_2}$, $\Gamma^{\delta_1, 0}$, and $\Gamma^{0, \delta_2}$,
	admits a uniform sunspot $\ep$-equilibrium.
	Moreover, in that game,
	there is a continue mixed action profile $c \in C$,
	such that at each stage of the sunspot $\ep$-equilibrium,
	at most one player $i$ quits with positive probability,
	and does so with probability at most $\ep$,
	while all other players follow $c_{-i}$.
\end{lemma}
\begin{proof}
	Let $\delta_1, \delta_2 \in (0, 1]$.
	Since $\Gamma = (I, (C_i)_{i \in I}, (Q_i)_{i \in I}, P, u)$ is a QL game,
	according to Observation \ref{observarion:best response matrix},
	at least one of the games $\Gamma^{\delta_1, \delta_2}$, $\Gamma^{\delta_1, 0}$, and $\Gamma^{0, \delta_2}$,
	admits a best response matrix $R$ which is a Q-matrix.
	Denote this game by $\Gamma'$.
	By Theorem \ref{theorem:Solan and Solan},
	the auxiliary game $\Gamma'$ admits a uniform sunspot $\ep$-equilibrium of the desired form.
\end{proof}

\begin{lemma}
	\label{lemma:Q-L sunspot}
	Let $\Gamma = (I, (C_i)_{i \in I}, (Q_i)_{i \in I}, P, u)$ be a QL game. Then $\Gamma$ admits a uniform sunspot $\ep$-equilibrium, for every $\ep > 0$.
\end{lemma}
\begin{proof}
	Fix $\ep > 0$.
	By Lemma \ref{lemma:Q sunspot},
	there are $\delta_1,\delta_2 \in [0,1]$ such that
	$\max \left\{ \delta_1, \delta_2 \right\} > 0$
	and the auxiliary game $\Gamma_{\delta_1,\delta_2} = (I, (C_i)_{i \in I}, (Q_i)_{i \in I}, P^{\delta_1, \delta_2}, u^{\delta_1, \delta_2})$
	admits a sunspot $\frac{\ep}{2}$-equilibrium $\xi$.
	Assume without lost of generality that $\delta_1>0$.
	The strategy profile $\xi$ is determined by
	a continue mixed action profile $c^\xi \in C$,
	such that at each stage of $\xi$,
	at most one player $i$ quits with positive probability,
	and does so with probability at most $\frac{\ep}{2}$,
	while all other players follow $c_{-i}^\xi$.
	
	More formally, $\xi$ has the following structure:
	Let $c^{\xi}$ be the continue mixed action profile from Lemma \ref{lemma:Q sunspot}.
	The continue mixed action of Player 2 in this action profile is $c^{\xi}_2 \coloneqq pc_2^1 + (1-p)c_2^2$, where if $\delta_2 >0$ then $p=0$. For every $t \in \dN$, a player $i_t \in I$ is chosen by the correlation device, alongside a quitting action $q_{i_t}^t \in Q_i$, an integer $M_t \in \dN$, and a deviation $\alpha_t \in (0, \frac{\ep}{2})$.
	For the next $M_t$ stages, player $i_t$ plays $\alpha_t q_{i_t}^t + (1 - \alpha_t) c^{\xi}_{i_t}$, while all other players play $c^{\xi}_{-i_t}$. The auxiliary game $\Gamma_{\delta_1, \delta_2}$ is absorbing during these $M_t$ stages with a probability of 
	$\rho_{\delta_1, \delta_2} (q_{i_t}^t, c^{\xi}, \alpha_t, M_t)
	\coloneqq 1 - (1 - \alpha_t \cdot P^{\delta_1, \delta_2}(q_{i_t}^t, c_{-i_t}^{\xi}))^{M_t}$. If the game is not absorb, then $i_{t+1} \in I$, $q_{i_{t+1}}^{t+1} \in Q_i$, $M_{t+1} \in \dN$, and $\alpha_{t+1} \in (0, \frac{\ep}{2})$ are chosen by the correlation device, and the process is continue.
	
	We will construct a sunspot strategy profile $\widehat{\xi}$ in $\Gamma$ that mimics $\xi$ as follows.
	The idea is that the strategy profile $\widehat{\xi}$ coincides with $\xi$ until it is Player 1's turn to play the quitting action $c_1^2$ with small probability.
	Then, we will expand this turn into many stages, and let Player 2 play $c_2^2$ with at least small probability as well.
	On one hand, Player 2 will play $c_2^2$ with small probability so that no other player can gain more than $\ep$ by deviating. On the other hand, we will repeat this play for many stages so that all the other players can monitor Player 2.
	The probability of Players 1 and 2 to play simultaneity $c_1^2$ and $c_2^2$ in these stages will be equal to the quitting probability of Player 1 quitting with $c_1^2$ in the original equilibrium of the auxiliary game.
	If $\delta_2>0$ as well, the symmetric process takes place when in $\widehat{\xi}$ it is Player 2 turn to quit with action $c_2^2$.
	
	Formally, the sunspot strategy profile $\widehat{\xi}$ has the same structure as $\xi$: For every $t \in \dN$, a player $\widehat{i}_t \in I$ is chosen by the correlation device, alongside a quitting action $q_{\widehat{i}_t}^t \in Q_i$, an integer $\widehat{M}_t \in \dN$, and a deviation $\widehat{\alpha}_t \in (0, \frac{\ep}{2})$.
	Player $\widehat{i}_t$ quits with probability $\widehat{\alpha}_t$ using $q_{\widehat{i}_t}^t$ for $\widehat{M}_t$, while all the other players play a continue action. If the game is not absorb, then $\widehat{i}_{t+1} \in I$, $q_{\widehat{i}_{t+1}}^{t+1} \in Q_i$, $\widehat{M}_{t+1} \in \dN$, and $\widehat{\alpha}_{t+1} \in (0, \frac{\ep}{2})$ are chosen by the correlation device, and the process is continue.
	Denote by $\rho (q_i, c, \alpha, M)
	\coloneqq 1 - (1 - \alpha \cdot P(q_i, c_{-i}))^{M}$
	the probability of the game $\Gamma$ to absorb when player $i$ plays the quitting action $q_i$ with probability $\alpha$ and the rest of the players play $c_{-i}$ for $M$ stages.
	Note that if $q_i \ne c_1^2, c_2^2$
	then $\rho (q_i, c_, \alpha, M) = \rho_{\delta_1, \delta_2} (q_i, c, \alpha, M)$,
	and that the payoff when this absorption occurs is equal,
	that is $u(q_i, c_{-2}) = u_{\delta_1, \delta_2}(q_i, c_{-2})$.
	
	The fashion of choosing $\widehat{i}_t$ and $q_{\widehat{i}_t}^t$ in $\widehat{\xi}$ is identical to the fashion of choosing $i_t$ and $q_{i_t}^t$ in $\xi$.
	If $q_{\widehat{i}_t}^t \ne c_1^2, c_2^2$, then $\widehat{M}_t$ and $\widehat{\alpha}_t$ are evaluate from $\widehat{i}_t$, $q_{\widehat{i}_t}^t$, and the correlation device that same way
	$M_t$ and $\alpha_t$ are evaluate from $i_t$, $q_{i_t}^t$, and the correlation device.
	
	We now construct $\widehat{M}_t$ and $\widehat{\alpha}_t$ out of $M_t$ and $\alpha_t$ when $q_{\widehat{i}_t}^t = c_1^2$.
	Let $p = max\left\{ c^{\xi}_2(c_2^2), \frac{\ep}{2} \right\}$, $\widehat{c}_2 \coloneqq pc_2^1 + (1-p)c_2^2$,
	and $\widehat{c}^{\xi, 2} \coloneqq (\widehat{c}_2, c^{\xi}_{-2})$.
	The function $\rho (c_1^2, \widehat{c}^{\xi, 2}, \alpha_t, M)$ is increasing on $M$.
	Then, there is an integer $\widehat{M}_t \ge M_t$
	such that $\widehat{M}_t$ is large enough to monitor between $\widehat{c}^{\xi, 2}_2$ and $(1-\ep)pc_2^1 + ((1-p) + \ep p) c_2^2$,
	and $\rho (c_1^2, \widehat{c}^{\xi, 2}, \alpha_t, \widehat{M}_t) \ge
	\rho_{\delta_1, \delta_2} (c_1^2, c^\xi, \alpha_t, M_t)$.
	$\rho (c_1^2, \widehat{c}^{\xi, 2}, \alpha, \widehat{M}_t)$ is decreasing on $\alpha$.
	Then, there is $0 < \widehat{\alpha}_t \le \alpha_t$, such that
	$\rho (c_1^2, \widehat{c}^{\xi, 2}, \widehat{\alpha}_t, \widehat{M}_t) =
	\rho_{\delta_1, \delta_2} (c_1^2, c^\xi, \alpha_t, M_t)$.
	Note that if $\widehat{\xi}$ is terminate after choosing quitting action $c_1^2$, then the payoff is $u(a^4) = u^{\delta_1, \delta_2}(a^3)$, which is the payoff if $\xi$ is terminate after choosing quitting action $c_1^2$.
	Therefore, if the correlation device choose quitting action $c_1^2$ at time $t$, the the sunspot strategy profiles $\xi$ and $\widehat{\xi}$ terminates at the same probability before time $t+1$, and yields the same payoff if they terminate.
	We repeat this calculation for quitting action $c_2^2$, if needed.
	
	Thus, we created a process given by $\widehat{\xi}$, which is similar to the process given by $\xi$ in the following way - 
	For every $t \in dN$, the probability of the process to be terminate after $t$ stages, is equal. Moreover, if the processes are terminated after time $t$, the payoff they yield are equal.
	For every player $i \ge 3$, the only new deviations are given due the change between $c^\xi$ to $\widehat{c}^{\xi, 2}$ and $\widehat{c}^{\xi, 1}$.
	But, $\abs{c^\xi - \widehat{c}^{\xi, 2}}, \abs{c^\xi - \widehat{c}^{\xi, 1}} \le \frac{\ep}{2}$, then any deviation of a player in $\widehat{\xi}$ is bounded by $\frac{\ep}{2} + \frac{\ep}{2}$.
	For Players $1$ and $2$, the only new deviation is not playing $\widehat{c}^{\xi, 2}_2$ and $\widehat{c}^{\xi, 1}_1$ respectively when needed. But, they are being monitored for lowering the probability for termination of the game by more then $\ep$, while they have no incentive to boost this probability (since they can quit in $\xi$ as well). Therefore, their deviation has no more then $\ep$ influence.
	Therefore, $\widehat{\xi}$ is a uniform sunspot $\ep$-equilibrium.
\end{proof}

\subsection{Additional Family of Auxiliary Games}
In this section we define a second family of auxiliary games,
which are similar to the auxiliary games $\Gamma^{\delta_1, \delta_2}$,
with an additional property that one of the players is restricted in the mixed action he can play.
\begin{definition}
	\label{def:auxiliary NQL}
	Let $\delta > 0$ and let $\alpha \in [0,1]$.
	Let $\Gamma = (I, (C_i)_{i \in I}, (Q_i)_{i \in I}, P, u)$ be an L-shaped game,
	and	let $\Gamma^{\delta, 0} = (I, (C_i)_{i \in I}, (Q_i)_{i \in I}, P^{\delta, 0}, u^{\delta, 0})$ be an auxiliary game of $\Gamma$.
	The auxiliary game $\Gamma^{\delta, 0}_\alpha$ is defined similarly to $\Gamma^{\delta, 0}$ with the following change:
	Player 2 cannot play the action $c_2^2$ with probability greater then $\alpha$.
\end{definition}
The auxiliary game $\Gamma^{0, \delta}_\alpha$ is defined
analogously to the auxiliary game $\Gamma^{\delta, 0}_\alpha$,
with Player 1 and action $c_1^2$ replacing the role of Player 2 and action $c_2^2$.
Both games $\Gamma^{0, \delta}_0$ and $\Gamma^{\delta, 0}_0$ are quitting games if $\delta > 0$, and general quitting games for every $\delta \ge 0$.

\begin{definition}
	\label{definition:chi}
	Let $\Gamma = (I, (C_i)_{i \in I}, (Q_i)_{i \in I}, P, u)$ be an absorbing game.
	Let $a \in A$ be an action profile and let $x \in \times_{i \in I} \Delta(A_i)$ be a mixed action profile.
	The \emph{per-stage probability of absorption by action profile $a$ under $x$ in $\Gamma$} is denoted by
	$$\chi(a,x) \coloneqq x(a) \cdot P(a) \period$$
	For every subset of action profiles $A' \subseteq A$, denote
	the \emph{per-stage probability of absorption by $A'$ under $x$ in $\Gamma$} by
	$$\chi(A', x) \coloneqq \sum_{a \in A'}\chi(a,x) \period $$
	For every mixed action profile $x$ with $P(x)>0$, the \emph{undiscounted payoff} under $x$ in $\Gamma$ is denote by
	$$\gamma(x) \coloneqq \sum_{a \in A} \dfrac{\chi(a,x)}{P(x)} u(a) \comma$$
\end{definition}
where $P(x) = \sum_{a \in A}\chi(a,x) = \chi(A, x)$
is the per-stage probability of absorption under $x$ in $\Gamma$.\\

The following lemma provides a condition for the existence of an almost stationary uniform $\ep$-equilibrium.
\begin{lemma}
	\label{lemma:alpha equilibrium}
	Let $\Gamma$ be a generic L-shaped game.
	For every $\ep > 0$ there exist $\delta_\ep, c_\ep > 0$,
	such that if $\delta < \delta_\ep$, $\alpha \in [0,1]$,
	and the mixed action $x$ is a stationary equilibrium of $\Gamma^{\delta, 0}_\alpha$
	that satisfies $0 < P^{1,0}(x) < c_\ep$,
	then $\Gamma$ admits an almost stationary uniform $\ep$-equilibrium,
\end{lemma}
\begin{proof}
	Let $\Gamma = (I, (C_i)_{i \in I}, (Q_i)_{i \in I}, P, u)$ be a generic L-shaped game.
	Fix $\ep > 0$. We will prove that $\Gamma$ admits a almost stationary uniform $18 \ep$-equilibrium.\\
	
	\textbf{Step 1: Notations}\\
	For every $\delta > 0$,
	let $\Gamma^{\delta, 0} = (I, (C_i)_{i \in I}, (Q_i)_{i \in I}, P^{\delta, 0}, u^{\delta, 0})$ be its auxiliary game.
	Denote by $\gamma^{\delta, 0}(x)$ the undiscounted payoff of this auxiliary game.
	In this proof, we will refer $c_1^2$ as a quitting action.
	That is, the quitting actions in the game are
	\[ Q'_i \coloneqq
	\begin{dcases*}
	Q_1 \cup \left\{c_1^2\right\} & $i = 1$,\\
	Q_i & $i \ge 2$,
	\end{dcases*}
	\]
	and the continue actions are
	\[ C'_i \coloneqq
	\begin{dcases*}
	\left\{c_1^1\right\} & $i = 1$,\\
	\left\{c_2^1, c_2^2\right\} & $i = 2$,\\
	\left\{c_3\right\} & $i \ge 3$.
	\end{dcases*}
	\]
	Thus, Player 2 is the only player to have two continue actions,
	while all other players have a single continue action.
	Denote all action profiles that contain a single quitting action by
	$$A^1 \coloneqq \left\{a \in A \mid
	\exists i \in I. a_i \in Q'_i, a_{-i} \in \times_{j \ne i}C'_j \right\} \period$$
	Denote $\widetilde{A}^1 \coloneqq A^1 \setminus \left\{a^3, a^4 \right\}$,
	and denote all action profiles that contain more then a single quitting action by
	$$A^{>1} \coloneqq A \setminus \left( A^1 \cup \left\{a^1, a^2 \right\} \right) \period$$
	Since $0 < P^{1,0}(x)$, it follows that $0 < P^{\delta,0}(x)$, and $x$ is absorbing in $\Gamma^{\delta, 0}$.
	Therefore,
	\begin{eqnarray*}
		\gamma^{\delta, 0}(x) &=&
		\sum_{a \in A^1} \dfrac{\chi^{\delta,0}(a,x)}{P^{\delta,0}(x)} u^{\delta,0}(a) +
		\sum_{a \in A^{>1}} \dfrac{\chi^{\delta,0}(a,x)}{P^{\delta,0}(x)} u^{\delta,0}(a) \\
		&=& \dfrac{\chi^{\delta,0}(a^3,x)}{P^{\delta,0}(x)} u^{\delta,0}(a^3) +
		\sum_{a \in A^1 \setminus \left\{a^3\right\}} \dfrac{\chi^{\delta,0}(a,x)}{P^{\delta,0}(x)} u^{\delta,0}(a) +
		\sum_{a \in A^{>1}} \dfrac{\chi^{\delta,0}(a,x)}{P^{\delta,0}(x)} u^{\delta,0}(a) \\
		&=& \dfrac{\chi^{\delta,0}(\left\{a^3, a^4\right\},x)}{P^{\delta,0}(x)} u(a^4) +
		\sum_{a \in \widetilde{A}^1} \dfrac{\chi(a,x)}{P^{\delta,0}(x)} u(a) +
		\sum_{a \in A^{>1}} \dfrac{\chi(a,x)}{P^{\delta,0}(x)} u(a) \period\\
	\end{eqnarray*}
	Define $p_{min} \coloneqq \min \left\{P(a) \mid a \in A, a \ne a^1, a^2, a^3 \right\}$.
	Since by assumption $P^{\delta,0}(x) < c_\ep$,
	it follows that for every player $i \in I$ and every quitting action $q_i \in Q'_i$, we have $x_i(q_i) < \frac{c_\ep}{p_{min}}$.
	We will define a mixed action profile such that it is an almost stationary uniform $18\ep$-equilibrium in $\Gamma$.\\
	
	\textbf{Step 2: Influence of the action profiles in $\bm{A^{>1}}$ on the undiscounted payoff is negligible}\\
	First, we prove that there is $c'_\ep$
	such that for every mixed action profile $y$,
	if $y_i(q_i) < c'_\ep$ for every quitting action $q_i$ of player $i$,
	for every $i \in I$,
	then $\dfrac{\chi(A^{>1},y)}{P(y)} < \ep$.\\
	Note that $1 - P(y) \ge (1 - c'_\ep)^{\abs{Q'}}$.
	Let $q_i \in Q'_i$ be a quitting action of player $i$,
	such that $q_i \ne c_1^2$.
	We can bound from below the probability of player $i$ to quit alone using $q_i$ by
	$$\chi(\left\{(q_i, a^1_{-i}), (q_i, a^2_{-i})\right\}, y) \ge
	y(q_i) \cdot (1 - c'_\ep)^{\abs{Q'}} \cdot p_{min} \period$$
	Denote the quitting action profile that includes player $i$ quitting with $q_i$ while other players quit as well by
	$$A(q_i) \coloneqq \left\{a \in A \mid
	\exists j \ne i. a_j \in Q'_j, a_i = q_i \right\} \period$$
	We can bound from above the probability of player $i$ to using $q_i$ alongside other players quitting by
	$$\chi(A(q_i), y) \le y(q_i) \cdot \frac{c'_\ep}{p_{min}} \cdot \abs{A} \period$$
	Therefore, if
	$c'_\ep \le \min \left\{1 - \dfrac{1}{\sqrt[\abs{Q'}]{2}},
	\dfrac{p_{min}^2 \cdot \ep}{2 \abs{A}}\right\}$,
	it follows that
	$$\chi(\left\{(q_i, a^1_{-i}), (q_i, a^2_{-i})\right\}, y) \cdot \ep \ge
	\chi(A(q_i), y) \period$$
	Since it is true for every $q_i \ne c_1^2 \in Q_i$, 
	we deduce $\chi(A^{>1},y) < \ep \cdot P(y)$, as we wanted.
	
	We can now approximate the undiscounted payoff of a mixed action profile with a small per-stage probability of absorbing.
	Let $\widetilde{x}$ be a strategy profile that satisfies
	$$P(\widetilde{x}) \le \min \left\{p_{min} \left( 1 - \dfrac{1}{\sqrt[\abs{Q'}]{2}} \right),
	\dfrac{p_{min}^3 \cdot \ep}{2 \abs{A}}\right\} \comma $$
	then $\chi(A^{>1},\widetilde{x}) < \ep \cdot P(\widetilde{x})$.
	We can conclude these two boundaries of $\gamma(\widetilde{x})$.
	The first is an upper bound,
	\begin{eqnarray}
	\gamma(\widetilde{x}) &=& \dfrac{\chi(\left\{a^3, a^4\right\},\widetilde{x})}
	{P(\widetilde{x})} u(a^4) +
	\sum_{a \in \widetilde{A}^1} \dfrac{\chi(a,\widetilde{x})}
	{P(\widetilde{x})} u(a)
	+ \sum_{a \in A^{>1}} \dfrac{\chi(a,\widetilde{x})}
	{P(\widetilde{x})} u(a)
	\nonumber \\
	&\le& \dfrac{\chi(\left\{a^3, a^4\right\},\widetilde{x})}
	{P(\widetilde{x})} u(a^4) +
	\sum_{a \in \widetilde{A}^1} \dfrac{\chi(a,\widetilde{x})}
	{P(\widetilde{x})} u(a)
	+\ep \nonumber \\
	&\le& \dfrac{\chi(\left\{a^3, a^4\right\},\widetilde{x})}
	{\chi(A^1,\widetilde{x})} u(a^4) +
	\sum_{a \in \widetilde{A}^1} \dfrac{\chi(a,\widetilde{x})}
	{\chi(A^1,\widetilde{x})} u(a)
	+ \ep	\label{eq:step2_le} \comma
	\end{eqnarray}
	and the second is a lower bound,
	\begin{eqnarray}
	\gamma(\widetilde{x}) &=& \dfrac{\chi(\left\{a^3, a^4\right\},\widetilde{x})}
	{P(\widetilde{x})} u(a^4) +
	\sum_{a \in \widetilde{A}^1} \dfrac{\chi(a,\widetilde{x})}
	{P(\widetilde{x})} u(a)
	+ \sum_{a \in A^{>1}} \dfrac{\chi(a,\widetilde{x})}
	{P(\widetilde{x})} u(a)
	\nonumber \\
	&\ge& \dfrac{\chi(\left\{a^3, a^4\right\},\widetilde{x})}
	{P(\widetilde{x})} u(a^4) +
	\sum_{a \in \widetilde{A}^1} \dfrac{\chi(a,\widetilde{x})}
	{P(\widetilde{x})} u(a)
	\nonumber \\
	&\ge& \dfrac{\chi(\left\{a^3, a^4\right\},\widetilde{x})}
	{\chi(A^1,\widetilde{x})} u(a^4) +
	\sum_{a \in \widetilde{A}^1} \dfrac{\chi(a,\widetilde{x})}
	{\chi(A^1,\widetilde{x})} u(a)
	- \ep \period \label{eq:step2_ge}
	\end{eqnarray}
	
	\textbf{Step 3: Variant of Step 2 for an action deviation}\\
	We will prove a similar estimate as in Step 2, for a different mixed action profile:
	Let $i \in I$ be a player, and let $q_i \in Q_i$ be a quitting action of player $i$.
	Let $y$ be a mixed action profile such that
	$y(q') < c'_\ep$ for every $q' \ne q_i \in Q$.
	We repeat the process as in Step 2.
	We know that
	$\chi(\left\{(q_i, a^1_{-i}), (q_i, a^2_{-i})\right\}, y) \ge
	y(q_i) \cdot (1 - c'_\ep)^{\abs{Q'}} \cdot p_{min}$ and
	$\chi(A(q_i), y) \le y(q_i) \cdot \frac{c'_\ep}{p_{min}} \cdot \abs{A}$.
	If
	$c'_\ep \le \min \left\{1 - \dfrac{1}{\sqrt[\abs{Q'}]{2}},
	\dfrac{p_{min}^2 \cdot \ep}{2 \abs{A}}\right\}$ we deduce that 
	\begin{equation}
	\label{eq:step3_1}
	\chi(\left\{(q_i, a^1_{-i}), (q_i, a^2_{-i})\right\}, y) \cdot \ep \ge
	\chi(A(q_i), y) \period
	\end{equation}
	Denote the quitting action profile that includes player $j$ quitting with $q_j$ while other players quit as well, but player $i$ does not quit using $q_i$, by
	$$A(q_j ; q_i) \coloneqq \left\{a \in A \mid
	\exists k \ne j. a_k \in Q'_k, a_j = q_j, a_i \ne q_i \right\} \period$$
	If
	$c'_\ep \le \min \left\{1 - \dfrac{1}{\sqrt[\abs{Q'}]{2}},
	\dfrac{p_{min}^2 \cdot \ep}{2 \abs{A}}\right\}$,
	we deduce that
	\begin{equation}
	\label{eq:step3_2}
	\chi(\left\{(q_i, a^1_{-i}), (q_i, a^2_{-i})\right\}, y) \cdot \ep \ge
	\chi(A(q_i), y)
	\end{equation}
	From both Eqs. (\ref{eq:step3_1}) and (\ref{eq:step3_2})
	we can deduce again that if
	$c'_\ep < \min \left\{1 - \dfrac{1}{\sqrt[\abs{Q'}]{2}},
	\dfrac{p_{min}^2 \cdot \ep}{2 \abs{A}}\right\}$
	then 
	$$\chi(A^{>1},y) < \ep \cdot P(y) \period$$
	Hence, if $\widetilde{x}$ is a mixed action profile
	such that
	$P(\widetilde{x}) < \min \left\{p_{min} \left( 1 - \dfrac{1}{\sqrt[\abs{Q'}]{2}} \right),
	\dfrac{p_{min}^3 \cdot \ep}{2 \abs{A}}\right\}$,
	and $q_i \in Q_i$ is a quitting action of player $i$ such that $q_i \ne c_1^2$,	then
	$$\chi(A^{>1},(q_i, \widetilde{x}_{-i})) < \ep \cdot P(q_i, \widetilde{x}_{-i}) \period$$
	Therefore $$\sum_{a \in A^{>1}} \dfrac{\chi(a,(q_i, \widetilde{x}_{-i}))}
	{P(q_i, \widetilde{x}_{-i})} u(a) < \ep \period$$
	From Step 2, we get that if $c_i \in C_i$ is a continue action, then
	$$\sum_{a \in A^{>1}} \dfrac{\chi(a,(c_i, \widetilde{x}_{-i}))}
	{P(c_i, \widetilde{x}_{-i})} u(a) < \ep \period$$
	We conclude that for every deviation $a_i \in A_i$ such that $a_i \ne c_1^2$,
	we have these two inequalities. The first is an upper bound of $\gamma(a_i, \widetilde{x}_{-i})$
	\begin{eqnarray}
	\gamma(a_i, \widetilde{x}_{-i}) &=&
	\dfrac{\chi(\left\{a^3, a^4\right\},(a_i, \widetilde{x}_{-i}))}
	{P(a_i, \widetilde{x}_{-i})} u(a^4) +
	\sum_{a \in \widetilde{A}^1} \dfrac{\chi(a,(a_i, \widetilde{x}_{-i}))}
	{P(a_i, \widetilde{x}_{-i})} u(a) 
	\label{eq:step3_le} \\
	&&+ \sum_{a \in A^{>1}} \dfrac{\chi(a,(a_i, \widetilde{x}_{-i}))}
	{P(a_i, \widetilde{x}_{-i})} u(a)
	\nonumber \\
	&\le& \dfrac{\chi(\left\{a^3, a^4\right\},(a_i, \widetilde{x}_{-i}))}
	{P(a_i, \widetilde{x}_{-i})} u(a^4) +
	\sum_{a \in \widetilde{A}^1} \dfrac{\chi(a,(a_i, \widetilde{x}_{-i}))}
	{P(a_i, \widetilde{x}_{-i})} u(a) +
	\ep \nonumber \\
	&\le& \dfrac{\chi(\left\{a^3, a^4\right\},(a_i, \widetilde{x}_{-i}))}
	{\chi(A^1,(a_i, \widetilde{x}_{-i}))} u(a^4) +
	\sum_{a \in \widetilde{A}^1} \dfrac{\chi(a,(a_i, \widetilde{x}_{-i}))}
	{\chi(A^1,(a_i,\widetilde{x}_{-i}))} u(a) +
	\ep \nonumber \comma
	\end{eqnarray}
	while the second is a lower bound of $\gamma(a_i, \widetilde{x}_{-i})$
	\begin{eqnarray}
	\label{eq:step3_ge}
	\gamma(a_i, \widetilde{x}_{-i}) &=&
	\dfrac{\chi(\left\{a^3, a^4\right\},(a_i, \widetilde{x}_{-i}))}
	{P(a_i, \widetilde{x}_{-i})} u(a^4) +
	\sum_{a \in \widetilde{A}^1} \dfrac{\chi(a,(a_i, \widetilde{x}_{-i}))}
	{P(a_i, \widetilde{x}_{-i})} u(a) \\
	&&+ \sum_{a \in A^{>1}} \dfrac{\chi(a,(a_i, \widetilde{x}_{-i}))}
	{P(a_i, \widetilde{x}_{-i})} u(a)
	\nonumber \\
	&\ge& \dfrac{\chi(\left\{a^3, a^4\right\},(a_i, \widetilde{x}_{-i}))}
	{P(a_i, \widetilde{x}_{-i})} u(a^4) +
	\sum_{a \in \widetilde{A}^1} \dfrac{\chi(a,(a_i, \widetilde{x}_{-i}))}
	{P(a_i, \widetilde{x}_{-i})} u(a) \nonumber \\
	&\ge& \dfrac{\chi(\left\{a^3, a^4\right\},(a_i, \widetilde{x}_{-i}))}
	{\chi(A^1,(a_i, \widetilde{x}_{-i}))} u(a^4) +
	\sum_{a \in \widetilde{A}^1} \dfrac{\chi(a,(a_i, \widetilde{x}_{-i}))}
	{\chi(A^1,(a_i,\widetilde{x}_{-i}))} u(a) -
	\ep \nonumber
	\end{eqnarray}
	
	\textbf{Step 4: Constructing an auxiliary mixed action profile $\bm{\widehat{x}^\delta}$}\\
	We define a mixed action profile $\widehat{x}^\delta$ as follows,
	\begin{eqnarray*}
		\widehat{x}^\delta_2(c_2^1) &\coloneqq& (1-\delta) \cdot x_2(c_2^1) \comma\\
		\widehat{x}^\delta_2(c_2^2) &\coloneqq& x'_2(c_2^2) + \delta \cdot x_2(c_2^1) \comma\\
		\widehat{x}^\delta_2(q_2) &\coloneqq& x_2(q_2) \quad \text{for every } q_2 \in Q_2 \comma \\
		\widehat{x}^\delta_2(q_2) &\coloneqq& x_2(q_2) \quad \text{for every } i \ne 2,\ a_i \in A_i \period
	\end{eqnarray*}
	In words, in $\widehat{x}^\delta$, Player 2 increases the probability of playing the action $c_2^2$ at the expense of the action $c_2^1$.
	Note that
	$\chi^{\delta, 0}(\left\{a^3, a^4 \right\}, x) = \chi(a^4, \widehat{x}^\delta)$.
	We conclude that if $c_\ep, \delta <
	\min \left\{\dfrac{p_{min}}{2} \left( 1 - \dfrac{1}{\sqrt[\abs{Q'}]{2}} \right),
	\dfrac{p_{min}^3 \cdot \ep}{4 \abs{A}}\right\}$,
	then $P(\widehat{x}^\delta) <
	\min \left\{p_{min} \left( 1 - \dfrac{1}{\sqrt[\abs{Q'}]{2}} \right),
	\dfrac{p_{min}^3 \cdot \ep}{2 \abs{A}}\right\}$.
	Therefore, Eqs. (\ref{eq:step2_le}) and (\ref{eq:step2_ge}) hold for the mixed action profile $\widehat{x}^\delta$,
	and Eqs. (\ref{eq:step3_le}) and (\ref{eq:step3_ge}) hold for every action deviation of $\widehat{x}^\delta$.\\
	Note that for every mixed action profile $y$,
	and every two action profiles $a, a' \in \widetilde{A}^1$
	such that $a_{-2} = a'_{-2}$,  $a_2 = c_2^1$, and $a'_2 = c_2^2$,
	we have
	\begin{equation}
	\left| \dfrac
	{\chi^{\delta, 0}(\left\{a, a'\right\}, y)}
	{\chi(\left\{a, a'\right\}, \widehat{y}^\delta)}
	- 1	\right| \le \delta \period \nonumber
	\end{equation}
	Therefore, we deduce that
	\begin{eqnarray}
	\label{eq:step4}
	\Bigg|&
	\dfrac{\chi(\left\{a^4\right\},\widehat{y}^\delta)}
	{\chi(A^1,\widehat{y}^\delta)} u(a^4) &+
	\sum_{a \in \widetilde{A}^1} \dfrac{\chi(a,\widehat{y}^\delta)}
	{\chi(A^1,\widehat{y}^\delta)} u(a)
	\\
	&- \dfrac{\chi^{\delta, 0}(\left\{a^3, a^4\right\},y)}
	{\chi^{\delta, 0}(A^1,y)} u(a^4) &-
	\sum_{a \in \widetilde{A}^1} \dfrac{\chi^{\delta, 0}(a,y)}
	{\chi^{\delta, 0}(A^1,y)} u(a) \Bigg| \le
	3 \delta \nonumber
	\end{eqnarray}
	
	\textbf{Step 5: Constructing an auxiliary mixed action profile $\bm{\widehat{x}^{\delta, \eta}}$}\\
	Let $\eta > 0$.
	We here define $\widehat{x}^{\delta, \eta}$, a mixed action profile which is a variant of $\widehat{x}^\delta$.
	Denote the maximal probability in which a player plays a quitting action under $x$ by
	$x_{max} \coloneqq \max \left\{x_i(q_i) \mid i \in I, q_i \in Q'_i \right\}$.
	We will define action profile which presents the relative weight of playing the quitting actions in $\cup_{i \in I} Q_i$, while ensuring that each quitting action is played with probability at most $\eta$.\\
	If $\eta \ge x_{max}$, then define
	$\widehat{x}^{\delta, \eta} \coloneqq \widehat{x}^\delta$.
	Otherwise, for every $i \in I$ such that $i \ne 2$, define
	\[
	\widehat{x}^{\delta, \eta}_i(a_i) \coloneqq
	\begin{dcases*}
	\frac{\eta}{x_{max}} \cdot \widehat{x}^\delta_i(a_i)
	& if $a_i \in Q'_i$,\\
	1 - \sum_{q_i \in Q'_i} \widehat{x}^{\delta, \eta}_i(q_i)
	& if $a_i \in C'_i$.
	\end{dcases*}
	\]
	For every quitting action $q_2 \in Q'_2$ of Player 2, define
	$$\widehat{x}^{\delta, \eta}_2(q_2) \coloneqq
	\frac{\eta}{x_{max}} \cdot \widehat{x}^\delta_2(q_2) \comma$$
	and for the continue actions $c_2^1$ and $c_2^2$,
	define $\widehat{x}^{\delta, \eta}_2$
	to satisfy the following two equations:
	\begin{eqnarray}
	\widehat{x}^{\delta, \eta}_2(c_2^1) + \widehat{x}^{\delta, \eta}_2(c_2^2) &=&
	1 - \sum_{q_2 \in Q'_2} \widehat{x}^{\delta, \eta}_2(q_2)
	\comma \label{eq:step5_1}\\
	\dfrac{\widehat{x}^{\delta, \eta}_2(c_2^1)}
	{\widehat{x}^{\delta, \eta}_2(c_2^1) + \widehat{x}^{\delta, \eta}_2(c_2^2)} &=&
	\dfrac{\widehat{x}^\delta_2(c_2^1)}
	{\widehat{x}^\delta_2(c_2^1) + \widehat{x}^\delta_2(c_2^2)}
	\period \label{eq:step5_2}
	\end{eqnarray}
	Eq. (\ref{eq:step5_1}) ensures that $\widehat{x}^{\delta, \eta}_2$ is a mixed action profile,
	while Eq. (\ref{eq:step5_2}) ensures that while Player 2 plays continue actions,
	the ratio of the probabilities to play the continue actions $c_2^1$ and $c_2^2$ under $\widehat{x}^\delta$ and under $\widehat{x}^{\delta, \eta}$ are the same.
	For every $\eta > 0$,
	every mixed action profile $y$,
	and every auxiliary mixed action profile $y^\eta$,
	we have
	\begin{eqnarray}
	&\dfrac{\chi(\left\{a^3, a^4\right\},y^\eta)}
	{\chi(A^1,y^\eta)} u(a^4) +
	\sum_{a \in \widetilde{A}^1} \dfrac{\chi(a,y^\eta)}
	{\chi(A^1,y^\eta)} u(a)
	=& \nonumber\\
	&=\dfrac{\chi(\left\{a^3, a^4\right\},y)}
	{\chi(A^1,y)} u(a^4) +
	\sum_{a \in \widetilde{A}^1} \dfrac{\chi(a,y)}
	{\chi(A^1,y)} u(a)& \label{eq:step5}
	\end{eqnarray}
	Note that if $c_\ep, \delta <
	\min \left\{\dfrac{p_{min}}{2} \left( 1 - \dfrac{1}{\sqrt[\abs{Q'}]{2}} \right),
	\dfrac{p_{min}^3 \cdot \ep}{4 \abs{A}}\right\}$,
	then Eqs. (\ref{eq:step2_le}), (\ref{eq:step2_ge}), (\ref{eq:step3_le}), (\ref{eq:step3_ge}), and (\ref{eq:step4}) hold for $\widehat{x}^{\delta, \eta}$.\\
	
	\textbf{Step 6: Statistical tests}\\
	We constructed a stationary strategy profile that is absorbing at every stage with low probability.
	Under this stationary strategy profile,
	Player 2 plays both actions $c_2^1$ and $c_2^2$ with positive probability,
	and he may profit by increasing the frequency in which he plays one of the actions at the expense of the other.
	To deter such detections,
	we add statistical tests to the strategy profile.\\
	If all players play non-absorbing stationary mixed action,
	they can verify whether a deviation by one of the players occur,
	using statistical test on the realized strategies.
	By \citet[Section 5.3]{solan_three-player},
	if stationary mixed action profile absorbs with a small probability in each turn,
	then this statistical test can still be uphold.
	That is, for our case,
	for stationary mixed action and every $\ep' > 0$,
	there is $\eta(\ep') > 0$ such that if the game absorb with probability smaller then $\eta(\ep')$ in each turn,
	then the players determine if player $i$ play the continue mixture $y_i$,
	or deviate from it by more then $\ep'$.
	Formally, for every $\ep' > 0 $, there is an integer $T_{\ep'} \in \dN$,
	such that if $\left\{x_m\right\}_{m \in \dN}$ are i.i.d. Bernoulli random variables with parameter $x$,
	then their average, after at least time $T_{\ep'}$,
	from $x$ is small enough.
	That is,
	$$\mathcal{P} \left( \left| \dfrac{\sum_{m=1}^{T} x_m}{T}
	- x \right| \ge \ep' \;
	\text{for some} \;
	T \ge T_{\ep'}
	\right) \le \ep' \period$$
	Let $\eta(\delta \cdot \ep) > 0$ be the constant related to $\delta \cdot \ep$,
	and set $\eta' \coloneqq \dfrac{\eta(\delta \cdot \ep)}{\abs{A}}$.
	Therefore, the players can identifying a deviation of Player 2 from $\widehat{x}^{\delta, \eta'}$, in scale of $\delta \cdot \ep$.\\
	Let $\eta \le \eta'$.
	Define by $\sigma^{\delta, \eta}$ a strategy similar to $\widehat{x}^{\delta, \eta}$,
	with the addition of statistical test by Players $1,3,4,\dots,\abs{I}$ whether the realized actions of Player 2 up to time $T$ are close to $\widehat{x}^{\delta, \eta}_2$, for every large enough $T$.
	We will show that
	if every player $i \in I$ cannot gain more then $\ep$ by deviating from
	$\sigma^{\delta, \eta}$
	to action $a_i \in A_i$ such then $a_i \notin C'_2$,
	then Player 2 cannot gain more then $2 \ep$ by changing the distribution of her continue actions.
	We assumed that every quitting action $q_2 \in Q_2$ is less than an $\ep$-efficient deviation against
	$\widehat{x}^{\delta, \eta}_{-2}$.
	Let $C_{-2} \coloneqq c_1^1 \times_{i \ge 3} c_i$ be the continue action profile of all players but Player 2.
	Since $P(\widehat{x}^{\delta, \eta}) < c_\ep$, we have
	\begin{eqnarray*}
		\gamma_2(\sigma^{\delta, \eta}) + \ep
		&\ge& \gamma_2(q_2, \widehat{x}^{\delta, \eta}_{-2}) \\
		&\ge& (1 - c_\ep) \cdot \gamma_2(q_2, C_{-2}) + c_\ep \cdot 0 \period
	\end{eqnarray*}
	Hence, for every $q_2 \in Q_2$,
	$$\gamma_2(q_2, C_{-2}) \le
	\frac{1}{1 - c_\ep} \gamma_2(\sigma^{\delta, \eta}) + \frac{\ep}{1 - c_\ep} \comma$$
	and therefore, if $c_\ep < \frac{\ep}{1+\ep}$ then
	\begin{equation*}
	\gamma(q_2, C_{-2}) \le
	\gamma_2(\sigma^{\delta, \eta}) + 3 \ep \period
	\end{equation*}
	Since $C_{-2}$ is a possible strategy of Players $1, 3, \dots, \abs{I}$,
	the min-max value of Player 2 is bounded by her best response to $C_{-2}$.
	The game $\Gamma$ is positive and recursive, and therefore
	\begin{equation*}
	\overline{v}_2(\Gamma) \le
	\max \left\{0,
	\gamma_2(\widehat{x}^{\delta, \eta}) + 3 \ep \right\}
	= \gamma_2(\sigma^{\delta, \eta}) + 3 \ep \period
	\end{equation*}
	We conclude that if $\eta < \eta'$, $c_\ep < \frac{\ep}{1+\ep}$,
	and every player $i \in I$ cannot gain more then $\ep$ by deviating from
	$\sigma^{\delta, \eta}$
	to action $a_i \in A_i$ such that $a_i \notin C'_2$,
	than Player 2 cannot gain more then $3\ep$ by deviating from $\sigma^{\delta, \eta}$,
	and it is a $3\ep$-equilibrium.
	Note that if $\kappa \in [1,\infty)$ is a positive constant,
	and every player $i \in I$ cannot gain more then $\kappa \ep$ by deviating from
	$\sigma^{\delta, \eta}$
	to action $a_i \in A_i$ such that $a_i \notin C'_2$,
	then $c_\ep < \frac{\ep}{1+\ep}$ implies that
	Player 2 cannot gain more then $3 \kappa \ep$ by deviating from $\sigma^{\delta, \eta}$,
	since $\frac{\ep}{1+\ep} \le \frac{\kappa \ep}{1+\kappa \ep}$.\\
	
	\textbf{Step 7: Constructing the almost stationary uniform $\bm{18 \ep}$-equilibrium}\\
	Let $\eta'$ be as defined in Step 6 and $\eta < \eta'$.
	Let $\delta <
	\min \left\{\dfrac{\ep}{3},
	\dfrac{p_{min}}{2} \left( 1 - \dfrac{1}{\sqrt[\abs{Q'}]{2}} \right),
	\dfrac{p_{min}^3 \cdot \ep}{4 \abs{A}}\right\}$
	and
	$c_\ep <
	\min \left\{\dfrac{\ep}{1+\ep},
	\dfrac{p_{min}}{2} \left( 1 - \dfrac{1}{\sqrt[\abs{Q'}]{2}} \right),
	\dfrac{p_{min}^3 \cdot \ep}{4 \abs{A}}\right\}$.
	We will prove that every player $i \in I$ cannot gain more then $6\ep$ by deviating from
	$\sigma^{\delta, \eta}$
	to action $a_i \in A_i$ such that $a_i \notin C'_2$,
	thus, together with Step 6, completing the proof.\\
	Fix $i \in I$, and an action $a_i \in A_i$ such that $a_i \notin C'_2$.
	We will prove that
	\begin{equation*}
	\gamma(a_i, \sigma^{\delta, \eta}_{-i}) \le
	\gamma(\sigma^{\delta, \eta}) + 6\ep \period
	\end{equation*}
	By definition, for every $a_i \in A_i$
	\begin{equation*}
	\gamma(a_i, \sigma^{\delta, \eta}_{-i}) \le
	\gamma(a_i, \widehat{x}^{\delta, \eta}_{-i}) \period
	\end{equation*}
	In Eq. (\ref{eq:step3_le}) we found an upper bound to $\gamma(a_i, \widehat{x}^{\delta, \eta}_{-i})$:
	\begin{equation*}
	\gamma(a_i, \widehat{x}^{\delta, \eta}_{-i}) \le
	\dfrac{\chi(\left\{a^3, a^4\right\},(a_i, \widehat{x}^{\delta, \eta}_{-i}))}
	{\chi(A^1,(a_i, \widehat{x}^{\delta, \eta}_{-i}))} u(a^4) +
	\sum_{a \in \widetilde{A}^1} \dfrac{\chi(a,(a_i, \widehat{x}^{\delta, \eta}_{-i}))}
	{\chi(A^1,(a_i,\widehat{x}^{\delta, \eta}_{-i}))} u(a) +
	\ep \period
	\end{equation*}
	We will use the equivalence of
	$\widehat{x}^{\delta, \eta}$ and $\widehat{x}^{\delta}$.
	By Eq. (\ref{eq:step5})
	\begin{eqnarray*}
		&&\dfrac{\chi(\left\{a^3, a^4\right\},(a_i, \widehat{x}^{\delta, \eta}_{-i}))}
		{\chi(A^1,(a_i, \widehat{x}^{\delta, \eta}_{-i}))} u(a^4) +
		\sum_{a \in \widetilde{A}^1} \dfrac{\chi(a,(a_i, \widehat{x}^{\delta, \eta}_{-i}))}
		{\chi(A^1,(a_i,\widehat{x}^{\delta, \eta}_{-i}))} u(a) +
		\ep \\
		&&\qquad
		= \dfrac{\chi(\left\{a^3, a^4\right\},(a_i, \widehat{x}^{\delta}_{-i}))}
		{\chi(A^1,(a_i, \widehat{x}^{\delta}_{-i}))} u(a^4) +
		\sum_{a \in \widetilde{A}^1} \dfrac{\chi(a,(a_i, \widehat{x}^{\delta}_{-i}))}
		{\chi(A^1,(a_i,\widehat{x}^{\delta}_{-i}))} u(a) +
		\ep \period
	\end{eqnarray*}
	Using Eq. (\ref{eq:step4}),
	We can switch between the games $\Gamma$ and $\Gamma^{\delta, 0}$.
	Since $3 \delta < \ep$, we get that
	\begin{eqnarray*}
		&&\dfrac{\chi(\left\{a^3, a^4\right\},(a_i, \widehat{x}^{\delta}_{-i}))}
		{\chi(A^1,(a_i, \widehat{x}^{\delta}_{-i}))} u(a^4) +
		\sum_{a \in \widetilde{A}^1} \dfrac{\chi(a,(a_i, \widehat{x}^{\delta}_{-i}))}
		{\chi(A^1,(a_i,\widehat{x}^{\delta}_{-i}))} u(a) +
		\ep \\
		&&\qquad
		\le \dfrac{\chi^{\delta,0}(\left\{a^3, a^4\right\},(a_i, x_{-i}))}
		{\chi^{\delta,0}(A^1,(a_i, x_{-i}))} u(a^4) +
		\sum_{a \in \widetilde{A}^1} \dfrac{\chi^{\delta,0}(a,(a_i, x_{-i}))}
		{\chi^{\delta,0}(A^1,(a_i,x_{-i}))} u(a) +
		2\ep \period
	\end{eqnarray*}
	We relate the result to $\gamma(a_i, x_{-i})$,
	through its lower bound presented in Eq. (\ref{eq:step3_ge}):
	\begin{eqnarray*}
		&&\dfrac{\chi^{\delta,0}(\left\{a^3, a^4\right\},(a_i, x_{-i}))}
		{\chi^{\delta,0}(A^1,(a_i, x_{-i}))} u(a^4) +
		\sum_{a \in \widetilde{A}^1} \dfrac{\chi^{\delta,0}(a,(a_i, x_{-i}))}
		{\chi^{\delta,0}(A^1,(a_i,x_{-i}))} u(a) +
		2\ep \\
		&&\qquad
		\le \gamma(a_i, x_{-i}) + 3\ep \period
	\end{eqnarray*}	
	Since the mixed action profile $x$ is an equilibrium in the game $\Gamma$,
	we get that
	\begin{equation*}
	\gamma(a_i, x_{-i}) + 3 \ep \le \gamma(x) + 3 \ep \period
	\end{equation*}
	By the upper bound of $\gamma(x)$, presented in Eq. (\ref{eq:step2_le}), we get
	\begin{equation}
	\gamma(x) + 3 \ep \le
	\dfrac{\chi^{\delta,0}(\left\{a^3, a^4\right\},x)}
	{\chi^{\delta,0}(A^1, x)} u(a^4) +
	\sum_{a \in \widetilde{A}^1} \dfrac{\chi^{\delta,0}(a,x)}
	{\chi^{\delta,0}(A^1, x)} u(a)
	+ 4 \ep \period \nonumber
	\end{equation}
	We use Eq. (\ref{eq:step4}), to switch again between the games $\Gamma$ and $\Gamma^{\delta,0}$
	\begin{eqnarray*}
		&&\dfrac{\chi^{\delta,0}(\left\{a^3, a^4\right\},x)}
		{\chi^{\delta,0}(A^1, x)} u(a^4) +
		\sum_{a \in \widetilde{A}^1} \dfrac{\chi^{\delta,0}(a,x)}
		{\chi^{\delta,0}(A^1, x)} u(a)
		+ 4 \ep \\
		&&\qquad
		\le
		\dfrac{\chi(\left\{a^3, a^4\right\},\widehat{x}^\delta)}
		{\chi(A^1, \widehat{x}^\delta)} u(a^4) +
		\sum_{a \in \widetilde{A}^1} \dfrac{\chi(a,\widehat{x}^\delta)}
		{\chi(A^1, \widehat{x}^\delta)} u(a)
		+ 5 \ep \period
	\end{eqnarray*}
	We use again the equivalence of
	$\widehat{x}^{\delta, \eta}$ and $\widehat{x}^{\delta, \eta}$.
	By Eq. (\ref{eq:step5})
	\begin{eqnarray*}
		&&\dfrac{\chi(\left\{a^3, a^4\right\},\widehat{x}^\delta)}
		{\chi(A^1, \widehat{x}^\delta)} u(a^4) +
		\sum_{a \in \widetilde{A}^1} \dfrac{\chi(a,\widehat{x}^\delta)}
		{\chi(A^1, \widehat{x}^\delta)} u(a)
		+ 5 \ep \\
		&&\qquad =
		\dfrac{\chi(\left\{a^3, a^4\right\},\widehat{x}^{\delta, \eta})}
		{\chi(A^1, \widehat{x}^{\delta, \eta})} u(a^4) +
		\sum_{a \in \widetilde{A}^1} \dfrac{\chi(a,\widehat{x}^{\delta, \eta})}
		{\chi(A^1, \widehat{x}^{\delta, \eta})} u(a)
		+ 5 \ep \period
	\end{eqnarray*}
	By the lower bound of $\gamma(x)$, presented in Eq. (\ref{eq:step2_ge}), we conclude that
	\begin{equation*}
	\dfrac{\chi(\left\{a^3, a^4\right\},\widehat{x}^{\delta, \eta})}
	{\chi(A^1, \widehat{x}^{\delta, \eta})} u(a^4) +
	\sum_{a \in \widetilde{A}^1} \dfrac{\chi(a,\widehat{x}^{\delta, \eta})}
	{\chi(A^1, \widehat{x}^{\delta, \eta})} u(a)
	+ 5 \ep \le
	\gamma(\widehat{x}^{\delta, \eta}) + 6 \ep \period
	\end{equation*}
	By definition
	\begin{equation*}
	\gamma(\widehat{x}^{\delta, \eta}) + 6 \ep =
	\gamma(\sigma^{\delta, \eta}) + 6 \ep \comma
	\end{equation*}
	and the result follows.
\end{proof}\\\\
Note that the same is true for the symmetric case of $\Gamma^{0, \delta}_\alpha$ and $P^{0,1}$:
for every $\ep > 0$ there exist $\delta_\ep, c_\ep > 0$,
such that if $\delta < \delta_\ep$, $\alpha \in [0,1]$,
and the mixed action $x$ is a stationary equilibrium of $\Gamma^{0, \delta}_\alpha$
that satisfies $0 < P^{0,1}(x) < c_\ep$,
then the game $\Gamma$ admits an almost stationary uniform $\ep$-equilibrium.

\subsection{NQL Games}
\label{section:non q}
In this section we study NQL games.
These games are similar to general quitting games
that satisfy the second condition in Theorem \ref{theorem:Solan and Solan}.
The following lemma was proven by \citet[Section 3.2, Case 3]{general_quitting}.
This lemma claims that the limit of discount equilibria cannot be a continue action, if the continue payoff is a witness of the best response matrix.
\begin{lemma}
	\label{lemma:lambda absorbing}
	Let $\Gamma = (I, (C_i)_{i \in I}, (Q_i)_{i \in I}, P, u)$ be a generic quitting game whose best response matrix is not a Q-matrix.
	Let $q \in \dR^{\abs{I}}$ be a witness for this matrix.
	Denote by $\Gamma(q) = (I, (C_i)_{i \in I}, (Q_i)_{i \in I}, P, u)$ the quitting game that it similar to $\Gamma$, except that the non-absorbing payoff is $q$; that is, $u(c) \coloneqq q$.
	For every $\lambda > 0$, let $x^\lambda$ be a $\lambda$-discounted stationary equilibrium of $\Gamma(q)$.
	Then $\lim_{\lambda \to 0} x^\lambda$ is absorbing in $\Gamma(q)$.
\end{lemma}

We will prove a similar version of this lemma, which claims that the limit of undiscounted equilibria in the auxiliary games cannot be the action profile $a^1$.
\begin{lemma}
	\label{lemma:auxiliary absorbing}
	Let $\Gamma = (I, (C_i)_{i \in I}, (Q_i)_{i \in I}, P, u)$ be an NQL game.
	Let $\delta: \dN \to [0,1]^2 \setminus \left\{(0,0)\right\}$ be a function such that
	$\lim_{n \to \infty} \delta(n) = \vec{0}$.
	Let $(x^n)_{n \in \dN}$ be a converging sequence of mixed action profiles such that $x^n$ is a stationary equilibrium in the auxiliary game $\Gamma^{\delta(n)}$ for every $n \in \dN$,
	and define $x^\infty \coloneqq \lim_{n \to \infty} x^n$.
	Then, $x^\infty \ne a^1$.
\end{lemma}
\begin{proof}
	Assume by contradiction that $x^\infty = a^1$.\\
	
	\textbf{Step 1: Defining auxiliary games}\\
	Since $\lim_{n \to \infty} \delta(n) = \vec{0}$,
	there exists a sequence $\mathcal{M} \coloneqq \left\{m_n\right\}_{n=1}^\infty$ of natural numbers,
	that satisfies one of the following conditions:
	\begin{itemize}
		\item[(M1)] The two sequences
		$\left\{ \delta_1(m_n) \right\}_{n=1}^\infty$ and 
		$\left\{ \delta_2(m_n) \right\}_{n=1}^\infty$
		are strictly decreasing and positive.
		That is, for every $n < n' \in \dN$, we have
		$0 < \delta_1(m_{n'}) < \delta_1(m_n)$ and $0 < \delta_2(m_{n'}) < \delta_2(m_n)$.
		\item[(M2)] The sequence
		$\left\{ \delta_1(m_n) \right\}_{n=1}^\infty$ is strictly decreasing and positive,
		while $\delta_2(m) = 0$ for every $m \in \mathcal{M}$.
		\item[(M3)] The sequence
		$\left\{ \delta_2(m_n) \right\}_{n=1}^\infty$ is strictly decreasing and positive,
		while $\delta_1(m) = 0$ for every $m \in \mathcal{M}$.
	\end{itemize}
	Condition (M3) is symmetric to condition (M2),
	hence we assume without loss of generality that we have a subsequence $\mathcal{M} \subseteq \dN$ of infinity size that satisfies either condition (M1) or condition (M2).\\
	Let $\Gamma^{\mathcal{M}} = (I, (C_i)_{i \in I}, (Q_i)_{i \in I}, P^{\mathcal{M}}, u^{\mathcal{M}})$ be the quitting game defined as follows:
	\begin{itemize}
		\item	$P^{\mathcal{M}}(a) = 1$ for every $a \ne a^1,a^2 \in A$,
		and $P^{\mathcal{M}}(a^1) = 0$,
		\item	If $M$ satisfies Condition (M1):
		$P^{\mathcal{M}}(a^2) = 1$, and $u^{\mathcal{M}} = u^{0.5, 0.5}$,
		\item	If $M$ satisfies Condition (M2):
		$P^{\mathcal{M}}(a^2) = 0$, and $u^{\mathcal{M}} = u^{0.5, 0}$.
	\end{itemize}
	By Observation \ref{observarion:best response matrix}, the best response matrices of $\Gamma^\mathcal{M}$
	coincide with the best response matrices of $\Gamma^{\delta(m)}$ for every $m \in \mathcal{M}$.
	Consequently, the game $\Gamma^\mathcal{M}$ is NQL game.\\
	
	\textbf{Step 2: Representing } $\bm{\lim_{m \to \infty} u^{\delta(m)}(x^{m})}$\\
	Recall that $P^{\delta(m)}(x^{m})$ is the probability under the strategy profile $x^{m}$ that the auxiliary game $\Gamma^{\delta(m)}$ terminates in a single stage.
	For every action profile $a \in A$,
	let $z(a) \coloneqq \lim_{m \to \infty} \dfrac{x^{m}(a) \cdot P^{\delta(m)}(a)}{P^{\delta(m)}(x^{m})}$ be the limit probability of absorption by action profile $a$ under $x^{m}$.
	We claim that
	\begin{eqnarray}
	\lim_{m \to \infty} u^{\delta(m)}(x^{m}) &=&
	\sum_{i \in I} \sum_{q_i \in Q_i} z(q_i, a^1_{-i}) \cdot u(q_i, a^1_{-i}) + \label{eq:aux 1}\\
	&&+ (z(c_1^2, c_2^1, a^1_{-1,2}) + z(c_1^1, c_2^2, a^1_{-1,2}) \nonumber \\
	&&+ z(c_1^2, c_2^2, a^1_{-1,2})) \cdot u(a^4) \period \nonumber
	\end{eqnarray}
	To show that Eq.~\eqref{eq:aux 1} holds, we recall that
	\begin{equation}
	\label{eq:Q1}
	u^{\delta(m)}(x^{m}) =
	\sum_{a \in A} \dfrac
	{x^{m}(a) \cdot P^{\delta(m)}(a) \cdot u(a)}
	{P^{\delta(m)}(x^{m})} \period
	\end{equation}
	We will show that $z(a) = 0$ for every action profile
	$a \not \in \left\{(q_i, a^1_{-i})\right\}_{q_i \in Q_i} \cup \left\{a^2, a^3, a^4\right\}$.
	Let then $a$ be such an action profile,
	and let $i_1,\dots,i_k$ be the players such that $a_{i_j} \in Q_{i_j}$ is a quitting action.
	We can assume without loss of generality that $(i_1,a_{i_1})  \ne (1, c_1^2),(2,c_2^2)$.
	Then $a_{i_1}\in Q_{i_1}$ and $P^{\delta(m)}(a^1_{-i_1}, a_{i_1})$ is independent of $m$ and $\delta$.
	Because $x^\infty = a^1$ is non-absorbing in $\Gamma^{\delta(m)}$, for every $i \in \left\{i_1,...i_k\right\}$, we have $\lim_{m \to \infty} x^{m}_i(a_i) = 0$. Hence,\footnotemark
	\begin{eqnarray*}
		&&x^{m}(a^1_{-i_1,\dots,-i_k},a_{i_1},\dots,a_{i_k}) \cdot  P^{\delta(m)}(a^1_{-i_1,\dots,-i_k},a_{i_1},\dots,a_{i_k}) \le \\
		&&\qquad \le x^{m}(a^1_{-i_1,\dots,-i_k},a_{i_1},\dots,a_{i_k}) \ll
		x^{m}(a^1_{-i_1},a_{i_1}) \cdot P^{\delta(m)}(a^1_{-i_1},a_{i_1}) \period
	\end{eqnarray*}
	It follows that $x^{m}(a^1_{-i_1,\dots,-i_k},a_{i_1},\dots,a_{i_k}) \ll x^{m}(a^1_{-i_1},a_{i_1})$, and the claim follows.\\
	\footnotetext{We denote $f(\delta) \ll g(\delta)$ whenever $\lim_{\delta \to 0}\frac{f(\delta)}{g(\delta)} = 0$.}
	
	\textbf{Step 3: Constructing best response action profiles}\\
	Since $x^{m}$ is a stationary equilibrium in $\Gamma^{\delta(m)}$,
	and since $x^\infty = a^1$, for every player $i$ and every action $a_i \ne a^1_i \in A_i$,
	if $x^{m}_i(a_i) > 0$ for every large enough $m$,
	then $a_i$ is a best response against $a^1_{-i}$ in the game $\Gamma^{\delta(m)}$.
	Therefore, for every player $i \in I$
	who satisfies $P^{\delta(m)}(x_i^{m}, a^1_{-i}) > 0$ for every $m$ sufficient large,
	every action profile $y_i$ that assigns a positive probability only to actions $a_i \in A_i$ that satisfy $x_i^{m}(a_i) > 0$ (for every $m$ sufficiently large),
	is a best response to $a^1_{-i}$ in $\Gamma^{\mathcal{M}}$,
	since they have the same best response matrices.	
	
	We now define such a mixed action profile, $y$.
	For every player $i \ne 1$,
	$$y_i(a_i) \coloneqq \lim_{m \to \infty}
	\dfrac{x_i^{m}(a_i) \cdot P^{\delta(m)}(a_i, a^1_{-i})}{P^{\delta(m)}(x_i^{m}, a^1_{-i})} \period$$
	Thus, $y_i$ represents the proportion probability of the player to quit with a specific action.
	For Player 1, we increase the probability to play the action $c_1^2$ by an amount that depends on the probability of the game to be absorb in $a^4$.
	If $x^{m}(a^4) > 0$ then $x^{m}_1(c_1^2) > 0$ and $x^{m}_2(c_2^2) > 0$.
	Let $P^{\delta(m)}_{1,2}(x^{m}) \coloneqq \dfrac{x^{m}(a^4) \cdot P^{\delta(m)}(a^4)}{P^{\delta(m)}(x^{m})}$.
	Therefore, if $P^{\delta(m)}(x_1^{m}, a^1_{-1}) + P^{\delta(m)}_{1,2}(x^{m}) > 0$ for every $m$ sufficient large, the following action profile is a best response of Player 1 against $a^1_{-1}$ in $\Gamma^{\mathcal{M}}$:
	\[
	y_1(a_1) = 
	\begin{dcases*}
	\lim_{m \to \infty} \dfrac{x_1^{m}(a_1) \cdot P^{\delta(m)}(a_1, a^1_{-1})}{P^{\delta(m)}(x_1^{m}, a^1_{-1}) + P^{\delta(m)}_{1,2}(x^{m})} & if $a_1 \in Q_1$,\\
	\lim_{m \to \infty} \dfrac{x_1^{m}(c_1^2) \cdot P^{\delta(m)}(a^3) + P^{\delta(m)}_{1,2}(x^{m})}{P^{\delta(m)}(x_1^{m}, a^1_{-1}) + P^{\delta(m)}_{1,2}(x^{m})} & if $a_1 = c_1^2$,\\
	0 & if $a_1 = c_1^1$.
	\end{dcases*}
	\]\\
	
	\textbf{Step 4: The contradiction}\\
	The matrix $R \coloneqq (r^i)_{i \in I}$ is a best response matrix in $\Gamma^{\mathcal{M}}$,
	and therefore it is not a Q-matrix.
	We will derive a contradiction by showing that the linear complementary problem $\lcp(R,q)$ has a solution,
	for every $q \in R^{\abs{I}}$.
	Let $z \coloneqq (z_1,\dots,z_{\abs{I}})$, where
	\[ z_i \coloneqq
	\begin{dcases*}
	\sum_{q_1 \in Q_1} y(q_1, a^1_{-i}) + y(a^3) + y(a^4) \comma & $i = 1$,\\
	\sum_{q_2 \in Q_2} y(q_2, a^1_{-i}) + y(a^2) \comma & $i = 2$,\\
	\sum_{q_i \in Q_i} y(q_i, a^1_{-i}) \comma & $i > 3$.
	\end{dcases*}
	\]
	If $z_1 = 0$, define $r^1 \coloneqq \vec{0}$. Otherwise, define
	\begin{eqnarray*}
		r^1 \coloneqq&&
		\sum_{q_1 \in Q_1} \lim_{m \to \infty}
		\dfrac{x_1^{m}(q_1) \cdot P^{\delta(m)}(q_1, a^1_{-1})}
		{P^{\delta(m)}(x_1^{m}, a^1_{-1}) +
			P^{\delta(m)}_{1,2}(x^{m})} \cdot u(q_1)\\
		&&+	\lim_{m \to \infty}
		\dfrac{x_1^{m}(c_1^2) \cdot P^{\delta(m)}(a^3) +
			P^{\delta(m)}_{1,2}(x^{m})}
		{P^{\delta(m)}(x_1^{m}, a^1_{-1}) +
			P^{\delta(m)}_{1,2}(x^{m})} u(a^4)
		\period
	\end{eqnarray*}
	If $z_2 = 0$, define $r^2 \coloneqq \vec{0}$. Otherwise, define
	\begin{eqnarray*}
		r^2 \coloneqq&&
		\sum_{q_2 \in Q_2} \lim_{m \to \infty}
		\dfrac{x_2^{m}(q_2) \cdot P^{\delta(m)}(q_2, a^1_{-2})}
		{P^{\delta(m)}(x_2^{m}, a^1_{-2})}
		\cdot u(q_2)\\
		&&+ \lim_{m \to \infty}
		\dfrac{x_2^{m}(c_2^2) \cdot P^{\delta(m)}(a^2)}
		{P^{\delta(m)}(x_2^{m}, a^1_{-2})}
		u(a^2) \period
	\end{eqnarray*}
	For every player $i \ne 1$, if $z_i = 0$, define $r^i \coloneqq \vec{0}$. Otherwise, define
	$$r^i \coloneqq  \sum_{q_i \in Q_i} \lim_{m \to \infty} \dfrac{x_i^{m}(q_i) \cdot P^{\delta(m)}(q_i, a^1_{-i})}{P^{\delta(m)}(x_i^{m}, a^1_{-i})} \cdot u(q_i)
	\period$$
	Set $w \coloneqq \lim_{m \to \infty} u_{\delta(m)}(x^{m})$ to be the limit payoff under $x^{m}$.
	Since the game is positive, $w \in \dR_+^{\abs{I}}$.
	We note that:
	\begin{enumerate}
		\item[(Q1)]	$w = Rz$, \label{equation:Q1}
		\item[(Q2)]	$z$ is a probability distribution:
		$z_i \ge 0$ for every $i \in I$,
		and $\sum_{i=1}^{\abs{I}} z_i = 1$, \label{equation:Q2}
		\item[(Q3)]	$w_i \ge R_{i,i}$ for every $i \in I$, \label{equation:Q3}
		\item[(Q4)]	if $z_i > 0$ then $w_i = R_{i,i}$ for every $i \in I$. \label{equation:Q4}
	\end{enumerate}
	Indeed, Condition (Q1) is given by Eq.~(\ref{eq:Q1}).
	Condition (Q2) follows from the definitions.
	Conditions (Q3) and (Q4) hold by the definition of equilibrium:
	for every $\delta$ sufficient small and for every $i \in I$,
	player $i$ cannot gain more than the equilibrium payoff while quitting alone, since the quitting probability of all other players is small, and quitting alone is an option for player $i$ in the game. If $z_i > 0$, then player $i$ quits with a positive probability and he is indifferent between quitting alone and the equilibrium payoff.
	Conditions (Q1)-(Q4) imply that $R$ is a Q-matrix, a contradiction.
\end{proof}\\

The next lemma asserts that if the limit of a sequence of stationary equilibria is absorbing in a sequence of games,
then the limit is an equilibrium in the limit game.
The claim is similar to Lemma 4 of \cite{absorbing},
hence its proof is omitted.
\begin{lemma}
	\label{lemma:absorbing equilibrium}
	Let $\Gamma = (I, (C_i)_{i \in I}, (Q_i)_{i \in I}, P, u)$ be an NQL game.
	Let $\delta: \dN \to [0,1]^2 \setminus \left\{(0,0)\right\}$ be a function such that
	$\lim_{n \to \infty} \delta(n) = \vec{0}$.
	Let $(x^n)_{n \in \dN}$ be a converging sequence of mixed action profiles such that $x^n$ is a stationary equilibrium in the auxiliary game $\Gamma^{\delta(n)}$ for every $n \in \dN$,
	and set $x^\infty \coloneqq \lim_{n \to \infty} x^n$.
	If $x^\infty$ is absorbing in $\Gamma$, that is, $P(x^\infty) > 0$, then $x^\infty$ is a stationary uniform $0$-equilibrium of $\Gamma$.
\end{lemma}

The following theorem, presented original by \cite{browder_1960} will be helpful during Lemma \ref{lemma:NQL main}.
\begin{theorem}[Browder, 1960]
	\label{Theorem:browder}
	Let $X \subseteq \dR^n$ be a convex and compact set,
	and let $F : [0,1] \times X \to X$ be a continuous function.
	Define $C_F := \{ (t,x) \in [0,1] \times X \colon x = f(t,x)\}$ be the set of fixed points of $f$.
	There is a connected component $T$ of $C_F$ such that $T \cap (\{0\} \times X) \neq \emptyset$
	and $T \cap (\{1\} \times X) \neq \emptyset$.
\end{theorem}
In the notations of Browder's Theorem, 
by Kakutani's Fixed Point Theorem, every set-valued function $F : X \to X$ with non-empty and convex values and closed graph
has at least one fixed point. 
Browder's Theorem states that when the set-valued function $F$ depends continuously on a one-dimensional parameter whose range
is $[0,1]$, 
the set of fixed points, as a function of the parameter, has a connected component whose projection to the set of parameters is $[0,1]$.\\

We are now ready to prove the main result of this section: every NQL game admits an almost stationary $\ep$-equilibrium.
\begin{lemma}
	\label{lemma:NQL main}
	Let $\Gamma = (I, (C_i)_{i \in I}, (Q_i)_{i \in I}, P, u)$ be an NQL game.
	Then for every $\ep > 0$, the game $\Gamma$ admits an almost stationary uniform $\ep$-equilibrium.
\end{lemma}
\begin{proof}
	By Lemma \ref{lemma:generic},
	we can assume without loss of generality that $\Gamma$ is a generic game.
	Fix $\ep > 0$.\\
	
	\textbf{Step 1: Construction of a continuous game-valued function}\\
	Since the game $\Gamma$ is an NQL game, the best response matrices of the games $\Gamma^{1,1}$, $\Gamma^{1,0}_0$, and $\Gamma^{0,1}_0$ are not Q-matrices.
	Let $q$, $q_1$ and $q_2$ be witnesses of these matrices (respectfully).
	Denote $\Theta \coloneqq [-1, 2]$, $\Omega \coloneqq (0,1]$,
	and $\mathcal{Q} \coloneqq \left\{q, q_1, q_2\right\}$.
	For every $\theta \in \Theta$ and every $\omega \in \Omega$, define
	\[ \Gamma_{\omega, \theta}(\mathcal{Q}) \coloneqq
	\begin{dcases*}
	\Gamma^{\omega, 0}_{1 + \theta} (-\theta q_1 + (1+\theta) q) \comma
	&$\theta \in [-1,0]$,\\
	\Gamma^{\theta \omega, (1 - \theta) \omega}(q) \comma & $\theta \in [0,1]$,\\
	\Gamma^{0, \omega}_{2 - \theta} ((\theta-1) q_2 + (2-\theta) q) \comma
	&$\theta \in [1,2]$.
	\end{dcases*}
	\]
	\begin{table}[h!]
		\begin{center}
			\begin{tabular}{y{65pt}|| y{90pt}| y{90pt}| y{15pt}}
				& $c_2^1$ & $c_2^2\ (\le 1 + \theta)$ & $q_2$ \\ [0.5ex] 
				\hline \hline
				$c_1^1$ & $- \theta q + (1 + \theta) q_1$ & $ - \theta q + (1 + \theta) q_1$ & * \\ 
				\hline
				$c_1^2$ & $u(a^4)\ ^{\omega}$* & $u(a^4)$ * & * \\
				\hline
				$q_1$ & * & * & * \\
			\end{tabular}
			\quad
			$-1 \le \theta \le 0$
			\quad
			\begin{tabular}{y{60pt}|| y{90pt}| y{90pt}| y{20pt}}
				& $c_2^1$ & $c_2^2$ & $q_2$ \\ [0.5ex] 
				\hline \hline
				$c_1^1$ & $q$ & $u(a^4)\ ^{(1-\theta) \omega}$* & * \\ 
				\hline
				$c_1^2$ & $u(a^4)\ ^{\theta \omega}$* & $u(a^4)$ * & * \\
				\hline
				$q_1$ & * & * & * \\
			\end{tabular}
		\quad
		$0 \le \theta \le 1$
		\quad
				\begin{tabular}{y{60pt}|| y{90pt}| y{90pt}| y{20pt}} 
				& $c_2^1$ & $c_2^2$ & $q_2$ \\ [0.5ex] 
				\hline \hline
				$c_1^1$ & $(2 - \theta) q + (\theta - 1) q_2$ & $u(a^4)\ ^{\omega}$* & * \\ 
				\hline
				$c_1^2\ (\le 2 - \theta)$ & $(2 - \theta) q + (\theta - 1) q_2$ & $u(a^4)$ * & * \\
				\hline
				$q_1$ & * & * & * \\
			\end{tabular}
		\quad
		$1 \le \theta \le 2$
			\caption{The auxiliary games
				$\Gamma_{\omega, \theta}(\mathcal{Q})$ for $\theta \in [-1,0]$ (top),
				for $\theta \in [0,1]$ (middle),
				and for $\theta \in [1,2]$ (bottom).}
			\label{table:auxiliary games of theorem NQL main}
		\end{center}
	\end{table}
	The function $(\omega, \theta) \mapsto \Gamma_{\omega, \theta}(\mathcal{Q})$
	maps an auxiliary game to each pair of parameters $(\omega, \theta) \in \Omega \times \Theta$.
	This function is continuous, in the sense that the absorption function and the payoff function change continuously with $\omega$ and $\theta$.
	For every $\omega \in \Omega$,
	in the game $\Gamma_{\omega, -1}(\mathcal{Q})$ Player 2 cannot play the action $c_2^2$.
	As $\theta$ increases from $-1$ to $0$,
	the probability by which Player 2 can play the action $c_2^2$ in the game $\Gamma_{\omega, \theta}(\mathcal{Q})$ increases to $1$,
	yet this action does not guarantee absorption.
	As $\theta$ increases from $0$ to $1$,
	the probability of absorption under $a^2$ (respectively $a^3$) increases to $\omega$ (respectively decreases to $0$).
	As $\theta$ increases from $1$ to $2$,
	the probability by which Player 1 can play the action $c_1^2$ decreases from $1$ to $0$.
	The non-absorbing payoff changes with $\theta$ as well:
	As $\theta$ increases to $0$, it changes linearly from $q_2$ to $q$;
	as $\theta$ increases from $0$ to $1$, it remains $q$;
	and as $\theta$ increases from $1$ to $2$, it changes linearly from $q$ to $q_1$.\\
	
	\textbf{Step 2: Applying Browder's Theorem}\\
	For every $\lambda > 0$, every $\omega \in \Omega$, and every $\theta \in \Theta$,
	let $\Xi^\lambda_\omega(\theta)$ be the set of all $\lambda$-discounted stationary equilibria of the auxiliary game $\Gamma_{\omega, \theta}(\mathcal{Q})$.
	Fix for a moment $\omega \in \Omega$ and $\lambda > 0$.
	Define
	$\mathcal{M}^\lambda_\omega \coloneqq \left\{ (\theta, \xi(\theta) \mid
	\theta \in \Theta, \xi(\theta) \in \Xi^\lambda_\omega(\theta) \right\}$.
	\cite{browder_1960} implies that
	the set $\mathcal{M}^\lambda_\omega$ has a connected component whose projection to the first coordinate is $\Theta$. Denote this connected component by $M^\lambda_\omega$.\\
	Define $M^0_\omega \coloneqq \limsup_{\lambda \to 0} M^\lambda_\omega
	= \cap_{\lambda > 0} \cup_{\lambda' < \lambda} M^{\lambda'}_\omega $,
	and  $M^0_0 \coloneqq \limsup_{\omega \to 0} M^0_\omega
	= \cap_{\omega > 0} \cup_{\omega' < \omega} M^0_{\omega'}$.
	The sets $(M^\lambda_\omega)_{\lambda, \omega}$, $(M^0_\omega)_{\omega}$, and $M^0_0$
	are all semi-algebraic sets,
	hence the sets $(M^0_\omega)_{\omega}$ are connected
	and their projection to the first coordinate is $\Omega$.
	Consequently, the same holds for the set $M^0_0$.\\

	\textbf{Step 3: Notations}\\
	Since $(M^0_\omega)_{\omega}$ and $M^0_0$ are connected sets whose projection to the first coordinate is $\Omega$,
	there exist continuous semi-algebraic functions
	$(\theta_\omega^0)_{\omega}, \theta_0^0 : [0,1] \to \Theta$,
	and $(\xi_\omega^0)_{\omega}, \xi_0^0 : [0,1] \to (\times_{i \in I} \Delta(A_i))$,
	such that
	\begin{itemize}
		\item[(a)] the image $\theta_0^0$ and $\theta_\omega^0$ for every $\omega \in \Omega$ is $\Theta$,
		\item[(b)] $(\theta_0^0(\beta), \xi_0^0(\beta)) \in M^0_\omega$ for every $\beta \in [0,1]$,
		\item[(c)] $(\theta_\omega^0(\beta), \xi_\omega^0(\beta)) \in M^0_\omega$ for every $\beta \in [0,1]$ and every $\omega \in \Omega$.
	\end{itemize}
	Without loss of generality,
	we can assume that for every $\omega \in \Omega$ we have
	$\theta^0_\omega(0), \theta^0_0(0) = -1$
	and $\theta^0_\omega(1), \theta^0_0(1) = 2$.
	Denote by $\Gamma(\omega, \beta)$ the auxiliary game $\Gamma_{\omega, \theta_{\omega}(\beta)}(\mathcal{Q})$.\\
	We divide the interval $[0,1]$ into smaller intervals, 
	which represent the different type of the auxiliary game admitted by $\Gamma(\omega, \beta)$.
	This is done as follows.
	Define recursively
	\begin{eqnarray*}
	\tau_1 &\coloneqq& -1 \comma \\
	\rho_i &\coloneqq& \inf \left( \left\{ \beta \mid
	\beta > \tau_i, 0 < \theta^0_0(\beta) < 1 \right\} \cup \left\{ 1 \right\} \right) \comma \\
	\tau_{i+1} &\coloneqq& \inf \left\{ \beta \mid 
	\beta > \rho_i, \theta^0_0(\beta) \notin (0,1) \right\} \period
	\end{eqnarray*}
	The definition ends when $\rho_i = 1$.\\
	Since $\theta^0_0(\cdot)$ is a semi-algebraic function,
	it enters the interval $(0,1)$ finitely many times.
	Therefore, there is an integer $i \in \dN$ such that $\rho_i = 1$.
	If $\beta \in (\rho_i, \tau_{i+1})$,
	then for small enough $\omega$,
	the auxiliary game $\Gamma(\omega, \beta)$
	has the structure of the game $\Gamma^{\delta_1, \delta_2}$,
	where $\delta_1 = \theta_{\omega}(\beta) \cdot \omega$ and $\delta_2 = (1-\theta_{\omega}(\beta)) \cdot \omega$.
	If $\beta \in (\tau_i, \rho_i)$,
	then for small enough $\omega$,
	the auxiliary game $\Gamma(\omega, \beta)$
	has the structure of either the game $\Gamma^{\delta, 0}_\alpha$ or the game $\Gamma^{0, \delta}_\alpha$, for some $\delta$ and $\alpha$.\\
	
	\textbf{Step 4: Properties of $\bm{\xi^0_0(\beta)}$ for $\bm{\rho_i \le \beta \le \tau_{i+1}}$}\\
	By the definition of $\tau$ and $\rho$,
	if $\rho_i < \beta < \tau_{i+1}$ then $0 < \theta_0^0(\beta) < 1$,
	and $\Gamma(\omega, \beta)$ is $\Gamma^{\delta_1, \delta_2}$,
	with $\delta_1 = \theta_{\omega}(\beta) \cdot \omega$ and $\delta_2 = (1-\theta_{\omega}(\beta)) \cdot \omega$.
	If $\beta \in \left\{\rho_i, \tau_{i+1}\right\}$,
	then $\theta_0^0(\beta) \in \left\{0,1\right\}$,
	and the auxiliary game $\Gamma(\omega, \beta)$ is either $\Gamma^{\delta, 0}$ or $\Gamma^{0, \delta}$,
	for some $\delta >0$.
	We will present some properties of both $\xi^0_\omega$ and $\xi^0_0$ in these cases.
	Lemma \ref{lemma:lambda absorbing} implies the following.
	\begin{claim}
	\label{claim:1}
		For every $\omega \in (0,1]$,
		if $\rho_i < \beta < \tau_{i+1}$
		then $\xi^0_\omega(\beta)$ is absorbing in $\Gamma(\omega,\beta)$.
	\end{claim}
	From Claim \ref{claim:1} we deduce that $\xi^0_\omega(\beta)$ is an equilibrium in $\Gamma(\omega,\beta)$ for every $\omega > 0$.
	Since $\lim_{\omega \to 0} \xi^0_\omega(\beta) = \xi^0_0(\beta)$,
	using Lemma \ref{lemma:auxiliary absorbing} we obtain the following result
	\begin{claim}
	\label{claim:2}
		For every $\rho_i < \beta < \tau_{i+1}$,
		we have $\xi^0_0(\beta) \ne a^1$.
	\end{claim}
	As the following claim asserts, Claim \ref{claim:2} extends to $\beta = \rho_i$ and $\beta = \tau_{i + 1}$.
	\begin{claim}
	\label{claim:3}
		$\xi^0_0(\rho_i), \xi^0_0(\tau_{i+1}) \ne a^1$.
	\end{claim}
	\begin{proof}
		We will argue that $\xi^0_0(\rho_i) \ne a^1$;
		the proof for $\xi^0_0(\tau_{i+1})$ is analogous.
		Since $\xi^0_\omega(\beta)$ is an equilibrium in $\Gamma(\omega,\beta)$
		for every $\omega > 0$ and every $\rho_i < \beta < \tau_{i+1}$,
		$\lim_{\omega \to 0^+} \xi^0_\omega(\rho_i + \omega) = \xi^0_0(\rho_i)$,
		the result follows from Lemma \ref{lemma:auxiliary absorbing}.
	\end{proof}\\\\
	Since $\xi^0_0(\beta) \ne a^1$ for every $\rho_i \le \beta \le \tau_{i+1}$,
	Claims \ref{claim:2} and \ref{claim:3}, together with Lemma \ref{lemma:absorbing equilibrium},
	yields the following.
	\begin{claim}
	\label{claim:real absorbing}
		For every $\rho_i \le \beta \le \tau_{i+1}$,
		if $\xi^0_0(\beta)$ is absorbing in $\Gamma$,
		then $\Gamma$ admits an almost stationary $\ep$-equilibrium.
	\end{claim}
	
	It is left to handle the case where
	$\xi^0_0(\beta) = (1-t_{\beta}) a^1 + t_\beta a^2$
	or $\xi^0_0(\beta) = (1-t_{\beta}) a^1 + t_\beta a^3$,
	for every $\beta \in [\rho_i, \tau_{i + 1}]$, where $t_\beta > 0$.
	The interval $[\rho_i, \tau_{i + 1}]$ is called \emph{a type 1 interval} if
	$\xi^0_0(\beta) = (1-t_{\beta}) a^1 + t_\beta a^2$ for every $\beta \in [\rho_i, \tau_{i + 1}]$.
	The interval $[\rho_i, \tau_{i + 1}]$ is called \emph{a type 2 interval} if
	$\xi^0_0(\beta) = (1-t_{\beta}) a^1 + t_\beta a^3$ for every $\beta \in [\rho_i, \tau_{i + 1}]$.
	We argue that the interval $[\rho_i, \tau_{i + 1}]$ is either a type 1 interval or a type 2 interval.
	Indeed, otherwise, there are
	$\beta, \beta' \in [\rho_i, \tau_{i + 1}]$,
	such that $\xi^0_0(\beta) = (1-t_{\beta}) a^1 + t_\beta a^2$
	and $\xi^0_0(\beta') = (1-t_{\beta'}) a^1 + t_\beta a^3$.
	The continuity of $\xi$ implies that one of the following two conditions holds:
	\begin{itemize}
		\item[($\iota$)] $\xi^0_0(\beta'') = a^1$ for some $\beta'' \in [\rho_i, \tau_{i + 1}]$,
		\item[($\iota \iota$)] $\xi^0_0(\beta'')$ is absorbing in the game $\Gamma$ for some $\beta'' \in [\rho_i, \tau_{i + 1}]$.
	\end{itemize}
	The former alternative is impossible due to Claim \ref{claim:2},
	while by Claim \ref{claim:real absorbing},
	the latter alternative implies that the game $\Gamma$ admits an almost stationary $\ep$-equilibrium.
	It follows that it is left to handle the case where $[\rho_i, \tau_{i + 1}]$ is either a type 1 interval or a type 2 interval.\\

	\textbf{Step 5: Properties of $\bm{\xi^0_0(\beta)}$ where $\bm{\tau_i \le \beta \le \rho_i}$}\\
	By the definition of $\tau$ and $\rho$,
	if $\tau_i \le \beta \le \rho_i$ then either
	$\theta_0^0(\beta) \le 0$,
	in which case $\Gamma(\omega, \beta) = \Gamma^{\omega, 0}_\alpha$ for some $\alpha \in [0,1]$;
	or $\theta_0^0(\beta) \ge 1$,
	in which case $\Gamma(\omega, \beta) = \Gamma^{0, \delta}_\alpha$
	for some $\alpha \in [0,1]$.
	\begin{claim}
	\label{claim:4}
		There exists $c_\ep' > 0$,
		such that for every $\beta \in [0,1]$,
		if $\theta^0_0(\beta) \le 0$
		and $0 < P^{0,1}(\xi^0_0(\beta)) < c_\ep'$,
		then $\Gamma$ admits an almost stationary $\ep$-equilibrium.
	\end{claim}
	\begin{proof}
	Let $c_\ep, \delta_\ep > 0$ be given by Lemma \ref{lemma:alpha equilibrium},
	and set $c_\ep' \coloneqq \frac{c_\ep}{2}$.
	If $0 < P^{0,1}(\xi^0_0(\beta)) < c_\ep'$,
	then there are $\beta'$ close to $\beta$ and $\omega < \delta_\ep$,
	such that $0 < P^{0,1}(\xi^0_\omega(\beta')) < c_\ep$.
	The result follows from Lemma \ref{lemma:alpha equilibrium}.
	\end{proof}\\\\
	An analogous result holds when $\theta^0_0(\beta) \ge 1$.
	\begin{claim}
	\label{claim:5}
		There exist $c_\ep > 0$,
		such that for every $\beta \in [0,1]$,
		if $\theta^0_0(\beta) \ge 1$
		and $0 < P^{1,0}(\xi^0_0(\beta)) < c_\ep$,
		then the game $\Gamma$ admits an almost stationary $\ep$-equilibrium.
	\end{claim}
	\begin{claim}
	\label{claim:side}
		Let $\beta, \beta' \in [\tau_i, \rho_i]$.
		If there are $t, t' \in (0,1]$ such that
		$\xi^0_0(\beta) = (1 - t) a^1 + t a^2$
		and $\xi^0_0(\beta') = (1 - t') a^1 + t' a^3$,
		then $\Gamma$ admits an almost stationary $\ep$-equilibrium.
	\end{claim}
	\begin{proof}
	Without loss of generality, $\theta^0_0(\beta'') \le 0$ for every $\beta'' \in [\tau_i, \rho_i]$.
	We know that $P^{0,1}(\xi^0_0(\beta)) > 0$ and $P^{0,1}(\xi^0_0(\beta')) = 0$.
	By the intermediate value theorem, there is $\beta''$ between $\beta$ and $\beta'$,
	such that $0 < P^{0,1}(\xi^0_0(\beta)) < c'_\ep$.
	The result follow from Claim \ref{claim:4}.
	\end{proof}\\\\		
	We deduce from Claim \ref{claim:side} that if $[\rho_i, \tau_{i+1}]$ is a type 1 interval
	(and therefore $\xi^0_0(\tau_{i+1}) = (1 - t) a^1 + t a^2$, for some $t > 0$),
	and $[\rho_{i+1}, \tau_{i+2}]$ is a type 2 interval,
	(and therefore $\xi^0_0(\rho_{i+1}) = (1 - t') a^1 + t' a^3$, for some $t' > 0$),
	then $\Gamma$ admits an almost stationary $\ep$-equilibrium.
	Hence, it left to handle the case where all intervals $[\rho_i, \tau_{i+1}]$ have the same type.\\
	
	\textbf{Step 6: The connected component ends}\\
	Without loss of generality, assume that all intervals $[\rho_i, \tau_{i+1}]$ have type 1.
	Then, $\xi^0_0(\rho_1) = (1 - t) a^1 + t a^2$ (where $t > 0$),
	and hence $P^{0,1}(\xi^0_0(\rho_1)) = 0$.
	If $\xi^0_0(0)$ is not a mixture of $a^1$ and $a^2$,
	then $P^{0,1}(\xi^0_0(0)) > 0$.
	Therefore, for every $c > 0$, there is $\beta \in [\tau_1, \rho_1]$ such that $0 < P^{0,1}(\xi^0_0(\beta)) < c$.
	By Claim \ref{claim:4}, the game $\Gamma$ admits an almost stationary $\ep$-equilibrium.\\
	If $\xi^0_0(0)$ is a mixture of $a^1$ and $a^2$,
	then, since $\theta^0_0(0) = -1$,
	we deduce that $\xi^0_0(0) = a^1$.
	We will prove that this case is not possible, and conclude the proof.
	As in Claim \ref{claim:1}, one can show that $\xi^0_\omega(0) \neq a^1$ for every $\omega > 0$.
	Therefore $\xi^0_\omega(0)$ are equilibria in $\Gamma(\omega,0)$ for every $\omega > 0$.
	The game $\lim_{\omega \to 0} \Gamma(\omega, 0)$ is a general quitting game
	with a single player who has two continue actions,
	while the other players have a single continue action.
	Since $\lim_{\omega \to 0} \xi^0_\omega(0) = a^1$,
	this is a contradiction, as in Claim \ref{claim:3}.
\end{proof}

\subsection{Uniform Sunspot $\ep$-equilibrium}
\label{subsection:L-shaped proof}
In this section we prove Lemma \ref{lemma:L-shaped main},
which state that every positive recursive L-shaped game admits a uniform sunspot $\ep$-equilibrium,
thus completing the proof of Theorem \ref{theorem:main}.\\
\begin{proof}
	Let $\ep>0$, and let $\Gamma$ be an L-shaped game. $\Gamma$ is either NQL game or QL game.
	If $\Gamma$ is an NQL game, then then by Lemma \ref{lemma:NQL main}, it admits an almost uniform $\ep$-equilibrium.
	If $\Gamma$ is a QL game, then by Lemma \ref{lemma:Q-L sunspot}, it admits a uniform sunspot $\ep$-equilibrium.
\end{proof}

\section{Discussion}
\label{section:discussion}
The problem that inspired this paper is whether every stochastic game admits a uniform sunspot $\ep$-equilibrium.
\cite{general_quitting} proved that every general quitting game admits such an equilibrium,
and in this paper, we proved the result to another family of absorbing games.
To do so, we developed two techniques that were used in this paper.\\
The first technique, which was introduced in Section \ref{section:spotted},
consists of dividing the set of non-absorbing mixed action profiles to equivalence classes.
Two mixed action profiles, $x$ and $y$, are in the same equivalence class
if there is $i$ such that $x_{-i} = y_{-i}$.
In spotted games, each equivalence class is composed of a single action profile.
We then considered, for each equivalence class,
an auxiliary game where all mixed action profiles that do not belong to that class are absorbing.
We then showed that a uniform sunspot $\ep$-equilibrium in such an auxiliary game,
in which the players play mainly a non-absorbing mixed action profile is a uniform sunspot $\ep$-equilibrium in the original game.
We also proved that if all these auxiliary games do not admit a uniform sunspot $\ep$-equilibrium,
in which the players play mainly a non-absorbing mixed action profile,
then a uniform absorbing stationary $0$-equilibrium exists.
The second technique, which was introduced in Section \ref{section:L-shaped},
consists of two families of auxiliary games,
where the limit of $\lambda$-discounted $\ep$-equilibrium in the auxiliary games allow us to explore and find uniform sunspot $\ep$-equilibrium in the original game.\\
We believe that by combining these two techniques,
one can proved the existence of a uniform sunspot $\ep$-equilibrium in more general families of absorbing games,
specifically in quitting absorbing games, that were introduced in Definition \ref{definition:quitting absorbing game}.

\bibliographystyle{chicago}
\bibliography{sunspot_eq_in_2d_games}

\begin{thebibliography}{}

\bibitem[\protect\citeauthoryear{Browder}{Browder}{1960}]{browder_1960}
Browder, F.~E. (1960).
\newblock On continuity of fixed points under deformations of continuous
  mappings.
\newblock {\em Summa Brasiliensis Mathematicae\/}~{\em 4}, 183--191.

\bibitem[\protect\citeauthoryear{Fink}{Fink}{1964}]{fink}
Fink, A.~M. (1964).
\newblock Equilibrium in a stochastic $n$-person game.
\newblock {\em Journal of Science of the Hiroshima University, Series A-I
  (Mathematics)\/}~{\em 28\/}(1), 89--93.

\bibitem[\protect\citeauthoryear{Flesch, Thuijsman, and Vrieze}{Flesch
  et~al.}{1997}]{cyclic_1997}
Flesch, J., F.~Thuijsman, and K.~Vrieze (1997).
\newblock Cyclic {Markov} equilibria in stochastic games.
\newblock {\em International Journal of Game Theory\/}~{\em 26\/}(3), 303--314.

\bibitem[\protect\citeauthoryear{Kohlberg}{Kohlberg}{1974}]{kohlberg_absorbing}
Kohlberg, E. (1974).
\newblock Repeated {Games} with {Absorbing} {States}.
\newblock {\em The Annals of Statistics\/}~{\em 2\/}(4), 724--738.

\bibitem[\protect\citeauthoryear{Mertens and Neyman}{Mertens and
  Neyman}{1981}]{mertens_neyman}
Mertens, J.~F. and A.~Neyman (1981).
\newblock Stochastic games.
\newblock {\em International Journal of Game Theory\/}~{\em 10\/}(2), 53--66.

\bibitem[\protect\citeauthoryear{Shapley}{Shapley}{1953}]{Shapley}
Shapley, L.~S. (1953).
\newblock Stochastic games.
\newblock {\em Proceedings of the National Academy of Sciences\/}~{\em 39},
  1095--1100.

\bibitem[\protect\citeauthoryear{Simon}{Simon}{2007}]{simon_2007}
Simon, R.~S. (2007).
\newblock The structure of non-zero-sum stochastic games.
\newblock {\em Advances in Applied Mathematics\/}~{\em 38\/}(1), 1--26.

\bibitem[\protect\citeauthoryear{Simon}{Simon}{2012}]{simon_2012}
Simon, R.~S. (2012).
\newblock A {Topological} {Approach} to {Quitting} {Games}.
\newblock {\em Mathematics of Operations Research\/}~{\em 37\/}(1), 180--195.

\bibitem[\protect\citeauthoryear{Simon}{Simon}{2016}]{simon_2016}
Simon, R.~S. (2016).
\newblock The challenge of non-zero-sum stochastic games.
\newblock {\em International Journal of Game Theory\/}~{\em 45\/}(1), 191--204.

\bibitem[\protect\citeauthoryear{Solan}{Solan}{1999}]{solan_three-player}
Solan, E. (1999).
\newblock Three-{Player} {Absorbing} {Games}.
\newblock {\em Mathematics of Operations Research\/}~{\em 24\/}(3), 669--698.

\bibitem[\protect\citeauthoryear{Solan and Solan}{Solan and
  Solan}{2018}]{general_quitting}
Solan, E. and O.~N. Solan (2018).
\newblock Sunspot {Equilibrium} in {General} {Quitting} {Games}.
\newblock {\em arXiv:1803.00878 [math]\/}.
\newblock arXiv: 1803.00878.

\bibitem[\protect\citeauthoryear{Solan and Solan}{Solan and
  Solan}{2019}]{quitting}
Solan, E. and O.~N. Solan (2019).
\newblock Quitting {Games} and {Linear} {Complementarity} {Problems}.
\newblock {\em Mathematics of Operations Research\/}, forthcoming.

\bibitem[\protect\citeauthoryear{Solan and Vieille}{Solan and
  Vieille}{2001}]{solan_vieille_2001}
Solan, E. and N.~Vieille (2001).
\newblock Quitting {Games}.
\newblock {\em Mathematics of Operations Research\/}~{\em 26\/}(2), 265--285.

\bibitem[\protect\citeauthoryear{Solan and Vieille}{Solan and
  Vieille}{2002}]{solan_vieille_correlated_2002}
Solan, E. and N.~Vieille (2002).
\newblock Correlated {Equilibrium} in {Stochastic} {Games}.
\newblock {\em Games and Economic Behavior\/}~{\em 38\/}(2), 362--399.

\bibitem[\protect\citeauthoryear{Solan and Vohra}{Solan and
  Vohra}{2001}]{solan_vohra_correlated_quitting_2001}
Solan, E. and R.~V. Vohra (2001).
\newblock Correlated {Equilibrium} in {Quitting} {Games}.
\newblock {\em Mathematics of Operations Research\/}~{\em 26\/}(3), 601--610.

\bibitem[\protect\citeauthoryear{Solan and Vohra}{Solan and
  Vohra}{2002}]{solan_vohra_correlated_2002}
Solan, E. and R.~V. Vohra (2002).
\newblock Correlated equilibrium payoffs and public signalling in absorbing
  games.
\newblock {\em International Journal of Game Theory\/}~{\em 31\/}(1), 91--121.

\bibitem[\protect\citeauthoryear{Takahashi}{Takahashi}{1964}]{takahashi_1964}
Takahashi, M. (1964).
\newblock Equilibrium points of stochastic non-cooperative \$n\$-person games.
\newblock {\em Journal of Science of the Hiroshima University, Series A-I
  (Mathematics)\/}~{\em 28\/}(1), 95--99.

\bibitem[\protect\citeauthoryear{Vieille}{Vieille}{2000}]{vieille_two-player_2000_B}
Vieille, N. (2000).
\newblock Two-player stochastic games {II}: {The} case of recursive games.
\newblock {\em Israel Journal of Mathematics\/}~{\em 119\/}(1), 93--126.

\bibitem[\protect\citeauthoryear{Vrieze and Thuijsman}{Vrieze and
  Thuijsman}{1989}]{absorbing}
Vrieze, O.~J. and F.~Thuijsman (1989).
\newblock On equilibria in repeated games with absorbing states.
\newblock {\em International Journal of Game Theory\/}~{\em 18\/}(3), 293--310.

\end{thebibliography}

\end{document}